\documentclass[a4paper,12pt,reqno]{amsart}
\usepackage{amsmath,amsfonts,amssymb,amsthm,enumerate,multicol}
\usepackage{mathrsfs}
\usepackage{tikz,graphicx}
\usepackage{hyperref}
\usepackage{caption,enumitem,wasysym}
\usepackage{cleveref}
\usepackage{epstopdf}
\usepackage{float,wrapfig}
\usepackage[labelsep=space]{caption}
\usepackage[font=footnotesize]{caption}
\usepackage{xcolor}
\DeclareFontShape{OT1}{cmr}{bx}{sc}{<->cmcsc10}{}
\makeatletter
\usepackage{verbatim}
\DeclareFontFamily{U}{wasy}{}
\DeclareFontShape{U}{wasy}{m}{n}{<->wasy10}{}
\DeclareFontShape{U}{wasy}{b}{n}{<->wasy10}{}
\makeatother
\usepackage{cases}
\usepackage[english]{babel}
\usepackage[autostyle]{csquotes}
\usepackage{cite}
\theoremstyle{plain}
\newtheorem{thm}{Theorem}[section]

\newtheorem{lem}[thm]{Lemma}
\newtheorem{prop}[thm]{Proposition}

\newtheorem{defn}[thm]{Definition}
\newtheorem{rem}[thm]{Remark}

\usepackage{mathtools}
\usepackage{amsmath,amsfonts,amssymb,amsthm,enumerate,multicol}
\usepackage{tikz}
\usepackage{float}
\usepackage{autobreak}

\allowdisplaybreaks

\numberwithin{equation}{section}

\begin{document} 
	\begin{center}
		\Large{\textbf{The Diffusive Exchange Driven Growth Model with unbounded kernels}}
	\end{center} 
	\medskip
	\medskip
	\centerline{${\text{ ${\text{Saumyajit Das$^{\dagger*}$}}$ and  Ram Gopal Jaiswal$^{\ddagger*}$}}$}\let\thefootnote\relax\footnotetext{$^{*}$Corresponding author.  \newline{\it{${}$ \hspace{.3cm} Email address: }}saumyajit.math.das@gmail.com (S. Das), maths.ram07@gmail.com (R. G. Jaiswal)}
	\medskip
	{\footnotesize

		\centerline{ ${}^{}$  $\dagger$ Department of Mathematics, Harish-Chandra Research Institute,}
		\centerline{ A CI of Homi
			Bhabha National Institute, Chhatnag Road,}
		\centerline{Jhunsi, Prayagraj 211019, India.}
		\centerline{ ${}^{}$  $\ddagger$ Department of Mathematics, Indian Institute of Technology Bombay,}
		\centerline{Powai, Mumbai, Maharashtra 400076, India}
		
	}
\bigskip
\begin{quote}
{\small {\em \bf Abstract.} We study the discrete diffusive exchange-driven growth (EDG) equations on a bounded smooth domain of arbitrary dimension subject to homogeneous Neumann boundary conditions. The system belongs to the class of infinite systems of semilinear partial differential equations with nonlinear source terms of quadratic type. Global-in-time existence of solutions is established for separable exchange kernels of the form \(K_{i,j}=b_i a_j\), where the donor rates exhibit at most linear growth while the receiver rates are sublinear. The analysis is based on a uniform Fisher information estimate obtained from an entropy-entropy dissipation identity. This estimate yields renormalized solutions to a truncated system with finitely many species. A compactness argument then enables passage to the limit in the exchange operator, leading to the existence of global-in-time solutions for the full system.
}
\end{quote}
\vspace{0.5cm}

\section{Introduction}\label{sec:intro}
Cluster growth phenomena arise naturally in a wide range of physical, chemical, and biological systems, and mathematical models describing the evolution of cluster-size distributions have received considerable attention over the past decades. Among the most extensively studied are the coagulation-fragmentation equations, which describe clusters undergoing binary merging and splitting events. Depending on the mechanism, fragmentation may occur spontaneously due to external forces, internal instabilities, or collisions between particles. The coagulation-fragmentation mechanism arises in many applications, including polymerization, aerosol science, cloud formation, and colloidal dynamics, and a substantial mathematical theory has been developed addressing existence, uniqueness, mass conservation, gelation, and long-time behavior \cite{bll2019}.

Another important mechanism of cluster growth is the exchange-driven growth (EDG) process \cite{Naim2003, IKR1998}. In contrast to coagulation–fragmentation dynamics, EDG describes the exchange of single-size units, called monomers, between clusters. More precisely, the interaction between a donor cluster of size \(i \geq 1\) and a receiver cluster of size \(j \geq 0\) is governed by the symbolic rule
\begin{align*}
\langle i \rangle \oplus \langle j \rangle \;\longrightarrow\; \langle i-1 \rangle \oplus \langle j+1 \rangle.
\end{align*}
This mechanism arises in several physical contexts and has been used to model social phenomena such as migration processes \cite{JZ2003}, population dynamics \cite{LR2002}, and wealth exchange \cite{IKR1998}. From a kinetic perspective, EDG may be viewed as a highly restricted collision-induced mass transfer process \cite{safronov1972, list1976}. Within the framework of collision-induced breakage equations, whose dynamics are governed by a collision kernel together with a daughter distribution describing the post-collision mass distribution, EDG is recovered by choosing a daughter distribution concentrated on this unique unit-mass transfer event.

\medskip
This article is devoted to the existence of weak solutions to the diffusive EDG equation, which describes evolution of the density of clusters of size $i$ at position $x \in \Omega$ and time $t \geq 0$, denoted by $f = (f_i(t,x))_{i \in \mathbb{N}\cup\{0\}}$, where $\Omega \subset \mathbb{R}^d$ is a bounded domain with smooth boundary $\partial\Omega$. The diffusive EDG equation reads as
\begin{equation}\label{main equation}
\begin{cases}
\partial_t f_i - d_i \Delta f_i = Q_i(f),
& \mbox{in}\ (0,\infty) \times \Omega,\; i \in \mathbb{N}\cup\{0\} \\[0.4em]
\nabla f_i \cdot n = 0,
&\mbox{on}\ (0,\infty) \times \partial\Omega,\; i \in \mathbb{N}\cup\{0\} \\[0.4em]
f_i(0,x) = f_{i,0}(x) \geq 0,
& \mbox{in}\ \Omega,\; i \in \mathbb{N}\cup\{0\},
\end{cases}
\end{equation}
where $n(x)$ denotes the unit outward normal at the point $x$ at the boundary $\partial\Omega$ and the exchange operator $Q_i(f)$ is given by
\begin{equation}\label{Q0 operator}
Q_0(f):= f_1 \sum_{j=0}^{\infty} K_{1,j} f_j
- f_0 \sum_{j=1}^{\infty} K_{j,0} f_j,
\end{equation}
\begin{align}\label{Qi operator}
Q_i(f):= \;& f_{i+1} \sum_{j=0}^{\infty} K_{i+1,j} f_j
+ f_{i-1} \sum_{j=1}^{\infty} K_{j,i-1} f_j \notag \\
& - f_i \sum_{j=0}^{\infty} K_{i,j} f_j
- f_i \sum_{j=1}^{\infty} K_{j,i} f_j, \qquad i\in \mathbb{N}.
\end{align}

In equation \eqref{main equation}-\eqref{Qi operator}, the rate at which a monomer is transferred from a donor of size \(i\) to a receiver of size \(j\) is represented by \(K_{i,j} \geq 0\) and $d_i$ represents the diffusion coefficient. Owing to the single-monomer exchange mechanism and the no-flux boundary condition, the EDG system \eqref{main equation} formally preserves both the total number of clusters and the total mass, that is,
\begin{equation}\label{conservation}
\sum_{i=0}^{\infty} \int_{\Omega} f_i(t,x)\, dx = \sum_{i=0}^{\infty} \int_{\Omega} f_{i,0}(x)\, dx,
\qquad
\sum_{i=1}^{\infty} i \int_{\Omega} f_i(t,x)\, dx = \sum_{i=1}^{\infty} i \int_{\Omega} f_{i,0}(x)\, dx,
\end{equation}
for all \(t \geq 0\). In \cite{Naim2003}, it is observed that conservation of the total mass may fail when \(K_{i,j}\) grows sufficiently fast, due to the formation of an infinite-size cluster in the case of a spatially homogeneous setting (i.e., without diffusion). For the spatially homogeneous EDG model with density \((f_i(t))_{i \in \mathbb{N}\cup\{0\}}\). The gelation time is defined by
\begin{equation}\label{gelation time}
T_{\mathrm{gel}} := \inf\!\left\{\, t \geq 0 \;:\; \sum_{i=1}^{\infty} i\, f_i(t) < \sum_{i=1}^{\infty} i\, f_{i,0} \,\right\}.
\end{equation}
Finite-time gelation refers to the case $0 < T_{\mathrm{gel}} < \infty$, while instantaneous gelation refers to $T_{\mathrm{gel}} = 0$. In the diffusive setting, the same notion is recovered by replacing $f_i(t)$ in \eqref{gelation time} with the spatial average $\int_\Omega f_i(t,x)\,dx$.

The first rigorous analysis of EDG models without diffusion is carried out in \cite{E2018}. The well-posedness of EDG models without diffusion is established under the assumption $K(j,k)\le Cjk$, and this result is extended to the class of symmetric kernels $K(j,k)\le C(j^\mu k^\nu+j^\nu k^\mu)$ with $\mu,\nu\le 2$, $\mu+\nu\le 3$. It is also shown that kernels growing like $K(j,k)\ge Cj^\beta$ with $\beta>2$ admit no solution. At the same time, in the intermediate regime $3<\mu+\nu\le 4$ only local existence has been shown, with finite-time gelation left as a conjecture. Later, in \cite{S2020, ES2021}, the well-posedness issue has been addressed for the large class of kernels under a more general class of initial data. Recently, the well-posedness for the class of kernels investigated in \cite{E2018} has been established in \cite{si2025} by removing the higher-moment limitation on the starting data. The gelation conjecture stated in \cite{E2018} has also been addressed to some extent. The long-time behavior of the spatially homogeneous EDG system has been investigated in \cite{EV2021} for a large class of kernels. Specifically, an explicit family of equilibrium solutions has been obtained under the assumption condition \(K(j,k)=b_j a_k\)  and also convergence to equilibrium has been shown.

These investigations are, however, confined to the spatially homogeneous setting in which cluster concentrations depend only on cluster size and time, and the evolution reduces to an infinite-dimensional system of ordinary differential equations. To the best of our knowledge, the discrete EDG system with spatial diffusion has not yet been investigated rigorously, which motivates the present work. The spatial extension of the coagulation–fragmentation equation provides a useful precedent: results on well-posedness, mass conservation, and long-time behavior have been obtained along lines that depart substantially from the homogeneous theory and rely on tools from the theory of reaction–diffusion systems, including duality estimates, entropy methods, and semigroup techniques \cite{wrzosek2002, laurenccot2002, laurencot2002continuous, amann2005local, walker2004discrete, walker2005new, canizo2010regularity, canizo2010absence}.

The EDG exchange operator is bilinear in the densities and couples infinitely many cluster sizes; in this respect, it resembles the coagulation term of a coagulation-fragmentation system. A structural feature on which much of the coagulation--fragmentation theory rests is, however, absent. In the coagulation equation, the equation governing the monomer density has a sign-definite negative right-hand side, since monomers cannot be created by coagulation; this yields an immediate a priori bound on the monomer density, from which the densities of larger clusters are controlled by induction on the cluster size. When fragmentation is added, the monomer equation is no longer purely a loss equation, but the argument can still be carried through provided the breakage rate is suitably controlled. In the EDG equation, however, this approach fails at the outset, since the exchange mechanism produces monomers whenever a cluster of size two sheds a particle, and the equation governing the monomer density is therefore not sign-definite even in principle.

The same structural obstruction was already encountered in the diffusive coagulation with nonlinear fragmentation in \cite{canizo2010regularity, canizo2010absence}, where it was overcome by means of an \(L^2\) estimate on the total mass density. This duality-based approach provides an alternative route to existence results both for nonlinear fragmentation models and for coagulation-fragmentation systems with diffusion, but it is restricted to the one-dimensional case. The same strategy can be adapted to the diffusive EDG equation in one space dimension under suitable assumptions on the kernel. In the present work, we focus on the existence of weak solutions in arbitrary space dimension, which calls for a different set of estimates. To this end, we assume that the exchange kernel
\[
(K_{i,j})_{i,j\in\mathbb{N}\cup\{0\}}
=(b_i a_j)_{i,j\in\mathbb{N}\cup\{0\}}
\]
is nonnegative and satisfies
\begin{align}\label{sectorwise bound of kernel}
    b_* (i+1)\leq b_i\leq b^*(i+1)
    \quad \text{and} \quad
    0<a_j\leq a^*(j+1)^{\alpha},
    \qquad
    \alpha\in(0,1),
    \quad
    \forall\, i,j\in \mathbb{N}\cup\{0\},
\end{align}
for some constants \(a^*, b^*, b_*>0\). As an immediate consequence of the above assumption, we have
\begin{equation}\label{H_kernel}
\sup_{i,j\in\mathbb{N}\cup\{0\}}\frac{K_{i,j}}{(i+1)(j+1)}<\infty,
\qquad
\lim_{j\to\infty}\frac{K_{i,j}}{(i+1)(j+1)}=0,
\quad
\text{for every } i\in\mathbb{N}\cup\{0\}.
\end{equation}

The diffusion coefficients \((d_i)_{i\in\mathbb{N}\cup\{0\}}\) are assumed to be uniformly bounded from above and below. More precisely, there exist constants \(0<d_*\le d^*<\infty\) such that
\begin{equation}\label{H_diff}
d_* \le d_i \le d^*,
\qquad
i\in\mathbb{N}\cup\{0\}.
\end{equation}

The initial data \((f_{i,0})_{i\in\mathbb{N}\cup\{0\}}\) are assumed to be nonnegative and satisfy
\begin{equation}\label{H_init}
\sum_{i\in\mathbb{N}\cup\{0\}}(i+1)\,f_{i,0}\in L^2(\Omega).
\end{equation}

In addition, we assume that
\begin{align}\label{def:entropy in}
\sum_{i\in\mathbb{N}\cup\{0\}}\int_{\Omega}
f_{i,0}\log\left(\frac{f_{i,0}}{\mathcal{Q}_i}\right)
<\infty,
\end{align}
where the sequence \((\mathcal{Q}_j)_{j\in\mathbb{N}\cup\{0\}}\) is defined recursively by
\[
\mathcal{Q}_0=1,
\qquad
\mathcal{Q}_j=\frac{a_{j-1}}{b_j}\,\mathcal{Q}_{j-1},
\qquad
j\ge1.
\]

It is worth mentioning that assumption \eqref{def:entropy in} is required to derive the entropy and Fisher information estimates (see \eqref{entropy diss ineq RTS}-\eqref{uniform bound on fisher info RTS} for details).

\medskip

Next, we introduce the notion of a weak solution to the diffusive EDG equation.
\begin{defn}[Weak solution]\label{definition weak solution}
Let $T>0$. A family $(f_i)_{i\in\mathbb{N}\cup\{0\}}$ with
\[
f_i \in L^1\!\left((0,T); W^{1,1}(\Omega)\right)  \cap L^2\bigl((0,T);L^2(\Omega)\bigr),
\qquad i\in\mathbb{N}\cup\{0\},
\]
is called a weak solution to the system \eqref{main equation} with non-negative initial datum
$(f_{i,0})_{i\in\mathbb{N}\cup\{0\}}$ satisfying \eqref{H_init} if, for every $i\in\mathbb{N}\cup\{0\}$ and all
\[
\xi \in C_c^\infty([0,T)\times\Omega),
\]
the following equality holds
\[
-\int_\Omega
\xi(0,x)\, f_{i,0}(x)\, dx
-
\int_0^T \int_\Omega
\partial_t \xi\, f_i
\,dx\,dt
+
d_i\int_0^T \int_\Omega
\nabla \xi \cdot \nabla f_i
\,dx\,dt
=
\int_0^T \int_\Omega
\xi \, Q_{i}(f)
\,dx\,dt,
\]
where \(Q_{i}\) is being defined in \eqref{Q0 operator}-\eqref{Qi operator}.
\end{defn}
Next, we state the main result of the paper.
\begin{thm}\label{main result}
Let $(f_{i,0})_{i\in\mathbb{N}\cup\{0\}}$ satisfy \eqref{H_init} and \eqref{def:entropy in}. Assume that the diffusion coefficients $(d_i)_{i\in\mathbb{N}\cup\{0\}}$ satisfy \eqref{H_diff} and that the kernel $(K_{i,j})_{i,j\in\mathbb{N}\cup\{0\}}$ satisfies \eqref{sectorwise bound of kernel}. Then the system \eqref{main equation} admits a global-in-time non-negative weak solution \((f_i)_{i\in\mathbb{N}\cup\{0\}}\) in the sense of Definition \ref{definition weak solution}.
\end{thm}
It is worth pointing out that there are two main difficulties in this problem. First, the number of unknowns is infinite. Second, the source term is nonlinear, being quadratic in nature. We begin by truncating the system with respect to the number of unknowns. For each \(N \in \mathbb{N}\),
\begin{equation}\label{truncated system}
\begin{cases}
\partial_t f_i^N - d_i \Delta f_i^N = Q_i^N(f^N)
& \mbox{in}\ (0,T)\times\Omega,\; 0\le i\le N \\[0.4em]
\nabla f_i^N\cdot n =0,
& \mbox{on}\ (0,T)\times\partial\Omega,\; 0\le i\le N \\[0.4em]
f_i^N(0,x)=f_{i,0}^N(x)\ge0,
& \mbox{in}\ \Omega,\; 0\le i\le N,
\end{cases}
\end{equation}
where the truncated exchange operators are given by
\begin{equation}\label{Q0N operator}
Q_0^N(f^N)
:=
f_1^N \sum_{j=0}^{N-1} K_{1,j}\, f_j^N
-
f_0^N \sum_{j=1}^{N} K_{j,0}\, f_j^N ,
\end{equation}
\begin{align}\label{QiN operator}
Q_i^N(f^N)
:=\;&
f_{i+1}^N \sum_{j=0}^{N-1} K_{i+1,j}\, f_j^N
+
f_{i-1}^N \sum_{j=1}^{N} K_{j,i-1}\, f_j^N
\notag\\
&-
f_i^N \sum_{j=0}^{N-1} K_{i,j}\, f_j^N
-
f_i^N \sum_{j=1}^{N} K_{j,i}\, f_j^N ,
\end{align}
for \(1\le i \le N-1\), and
\begin{equation}\label{QNN operator}
Q_N^N(f^N)
:=
-
f_N^N \sum_{j=0}^{N-1} K_{N,j}\, f_j^N
+
f_{N-1}^N \sum_{j=1}^{N} K_{j,N-1}\, f_j^N.
\end{equation}

Note that here as well, the source is nonlinear, being quadratic. We will construct a solution to the main system \eqref{main equation} as the limit, as \(N \to \infty\), of the solutions to the truncated system \eqref{truncated system}. Moreover, one can verify that $\displaystyle{\sum_{i=0}^{N} Q_i^N = 0 \ \text{and}\ \sum_{i=0}^{N} i\,Q_i^N = 0}$,
following arguments analogous to those in Lemma~\ref{mass and momentum with auxiliary equ}. Consequently, the truncated system \eqref{truncated system} inherits the total number of particles and mass conservation properties encoded in the structure of the source term. Note that the truncated system is a reaction-diffusion system with a total number of particles and a mass-conservation structure.

Reaction-diffusion equations with quadratic nonlinearities have been studied extensively in \cite{pierre2003weak, Pierre2010}. For a comprehensive discussion of reaction-diffusion systems, see \cite{canizo2014improved, desvillettes2015duality, desvillettes2007global, fitzgibbon2021reaction, fellner2021uniform, das2025existence}. In \cite{desvillettes2007global}, significant progress in the treatment of quadratic source terms in reaction-diffusion equations has been established. Exploiting only the mass conservation property, the authors proved that the source term belongs to \(L^1((0,T);L^1(\Omega))\). Let us state the result below.

\begin{thm}[\cite{desvillettes2007global}]\label{finite L^2 estimate}
    For $1\leq i\leq m$, let $\mathcal{F}_i\geq 0$ satisfies: for $1\leq i\leq m$
    \begin{equation*}
        \left \{
         \begin{aligned}
             \partial_t \left( \sum_{i=1}^m \mathcal{F}_i\right)- \Delta \left( \sum_{i=1}^m d_i\mathcal{F}_i\right)=&0 \qquad && \mbox{in}\ (0,T)\times\Omega\\
             \nabla \mathcal{F}_i\cdot n=& 0 \qquad && \mbox{on} \ (0,T)\times\partial\Omega\\
             \mathcal{F}_i(0,x)=& \mathcal{F}_{i}^{\rm{in}} \qquad && \mbox{in}\ \Omega.
         \end{aligned}
        \right.
    \end{equation*}
    Let $\mathcal{F}_i^{\rm in} \in L^2((0,T)\times\Omega)$ for $i = 1, \dots, m$.
Then
\begin{align*}
\int_{0}^T \int_{\Omega} \sum\limits_{i=1}^m \mathcal{F}_i \times \sum\limits_{i=1}^m d_i\mathcal{F}_i \, dx \, dt \leq C(\Omega,T) \max\limits_{1\leq i\leq m}\{d_i\} \times \sum_{i=1}^m \| \mathcal{F}_i^{\rm{in}}\|_{L^2(\Omega)},
\end{align*}
where the constant $C(\Omega,T)$ is positive and only depends on the domain and $T$.
\end{thm}
In the context of the discrete coagulation--fragmentation model, where the number of unknowns is infinite, an analogous estimate was established in \cite{canizo2010absence, das2026existence}. Combining this estimate with the approach developed in \cite{Pierre2010}, one can construct an \(L^1((0,T);W^{1,1}(\Omega))\) solution to the truncated system \eqref{truncated system} by passing to the limit in a suitable regularized version of the system (see \cite{Pierre2010} for further details).

Another difficulty arises when the truncation parameter \(N\) tends to infinity, that is, when passing to the limit \(N \to \infty\) in \eqref{truncated system}. The main challenge lies in handling two simultaneous limiting processes in the source term. The first concerns compactness with respect to the truncation index \(N\), while the second stems from the presence of the partial sum of an infinite sum in the source term. 
\newline
To overcome this difficulty, we truncate the range of the solution to the truncated system \eqref{truncated system} by means of a truncation-to-identity type map, as described in \eqref{truncation to identity}. Our goal is to analyze the behavior of the composition of such a truncation-to-identity function with the solution of the truncated system \eqref{truncated system}, by deriving the equation satisfied by the resulting composite function. The main obstacle here is the lack of sufficient regularity of the solutions to the truncated system \eqref{truncated system}. To overcome this issue, we revisit the techniques of renormalized solutions developed in \cite{JF2015}. In this article, we construct an \(L^1((0,T);W^{1,1}(\Omega))\) global-in-time weak nonnegative solution to the truncated system \eqref{truncated system} by employing the methods introduced in \cite{Pierre2010}. Furthermore, we show that this particular weak solution is also a renormalized solution of the same truncated system. This is achieved by exploiting the uniform bound on the Fisher information
\[
\sum_{i=0}^N \int_0^t \int_{\Omega} \frac{|\nabla f^N_i|^2}{f^N_i}\leq M_{\mathcal{F}},
\]
as established in Theorem \ref{existence of solution: truncated system}. This approach is motivated by the work in \cite{JF2015}. Once this step is achieved (as described in Proposition \ref{TS: composition with a special truncation to identity}), we can overcome the compactness issue in the nonlinear source term and ultimately obtain a global-in-time weak nonnegative solution to the main system \eqref{main equation} belonging to \(L^1((0,T);W^{1,1}(\Omega))\).
\newline
To construct such a global-in-time weak nonnegative solution to \eqref{truncated system}, which also serves as a renormalized solution, we use a damping in the the source term as described in \eqref{regularized truncated system}. We then apply the following theorem to guarantee the existence of a global-in-time strong solution to the damped system \eqref{regularized truncated system}. Proof of the theorem can be found in \cite{amann1995linear, quittner2007, rothe2006}.
\begin{thm}\label{existence theory for regular source}
    For $1\leq i\leq m$, let $d_i> 0$. Let $\mathcal{F}_i\geq 0$ satisfies: for $1\leq i\leq m$
    \begin{equation}\label{existence 1}
        \left \{
         \begin{aligned}
             \partial_t  \mathcal{F}_i- d_i \Delta \mathcal{F}_i=&\mathcal{A}_i(\mathcal{F}) \qquad && \mbox{in}\ (0,T)\times\Omega\\
             \nabla \mathcal{F}_i\cdot n=& 0 \qquad && \mbox{on} \ (0,T)\times\partial\Omega\\
             \mathcal{F}_i(0,x)=& \mathcal{F}_{i}^{\rm{in}} \qquad && \mbox{in}\ \Omega.
         \end{aligned}
        \right .
    \end{equation}
    Let \(\mathcal{F}_i^{\rm in} \in L^2((0,T)\times\Omega)\) for \(i=1,\dots,m\), and let \(\mathcal{A}_i\) be a globally Lipschitz continuous function for each \(i=1,\dots,m\). Then the system \eqref{existence 1} admits a unique global-in-time classical solution, which is smooth in \((0,T)\times\Omega\).
\end{thm}
Furthermore, we show that the solutions to the damped truncated system \eqref{regularized truncated system} are nonnegative. This follows from the nonnegativity of the initial data together with the following theorem.
\begin{thm}[Nonnegativity of the solutions, \cite{Pierre2010}]\label{nonnegativity of solution}
Let $\mu_i>0$ for $1\le i\le m$. Let $v_i:(0,T)\times\Omega\to\mathbb{R}$ be a strong solution of the system
\[
\left\{
\begin{aligned}
\partial_t v_i - \mu_i \Delta v_i &= \mathcal{Y}_i(v_1,\dots,v_m) && \text{in } (0,T)\times\Omega\\
\nabla v_i\cdot n &= 0 && \text{on } (0,T)\times\partial\Omega\\
v_i(0,x) &\ge 0 && \text{in } \Omega,
\end{aligned}
\right.
\]
for $1\le i\le m$. Assume that $\mathcal{Y}=(\mathcal{Y}_1,\dots,\mathcal{A}_m):\mathbb{R}^m\to\mathbb{R}^m$ is quasipositive, that is,
\begin{align}\label{quasipositive}
\mathcal{Y}_i(w_1,\dots,w_{i-1},0,w_{i+1},\dots,w_m) \ge 0,
\quad \forall (w_1,\dots,w_m)\in [0,+\infty)^m.
\end{align}
Then, $v_i\ge 0$ for all $i=1,\dots,m$.
\end{thm}
In this article, the Fisher information plays a crucial role in deriving several  estimates. In particular, through an entropy analysis, we establish a uniform bound on the Fisher information associated with the damped truncated system \eqref{regularized truncated system}, where the bound is uniform with respect to both the damping parameter and the truncation parameter `\(N\)'. In the study of reaction-diffusion systems, entropy methods play a fundamental role in analyzing the qualitative behavior of solutions. For further details, we refer the reader to \cite{desvillettes2007global, desvillettes2015duality, JF2015}. In the context of the damped truncated system \eqref{regularized truncated system}, we define the entropy functional by
\[
\sum_{i=0}^N \int_{\Omega} f_{i,\varepsilon}^N
\log \left(\frac{f_{i,\varepsilon}^N}{\mathcal{Q}_i}\right).
\]
Here, \(f_{i,\varepsilon}^N\), for \(0 \leq i \leq N\), denotes the solution to the damped truncated system \eqref{regularized truncated system} (here $\epsilon$ being the regularizing parameter), while the constants \(\mathcal{Q}_i\) are suitably chosen weights, whose definition and properties are described in Section \ref{sec:RTS}. In Section \ref{sec:RTS}, we show that the above entropy functional is non-increasing in time. Moreover, the Fisher information $\displaystyle{\sum_{i=0}^N \int_0^t \int_{\Omega}
\frac{|\nabla f_{i,\varepsilon}^N|^2}{f_{i,\varepsilon}^N}}$ naturally appears in the entropy dissipation identity as the contribution of the diffusion process to the entropy decay rate. Through the entropy-entropy dissipation identity, we obtain a uniform bound on the Fisher information that is independent of both the damping parameter \(\epsilon\) and the truncation parameter \(N\) (see \eqref{uniform bound on fisher info RTS} in Section~\ref{sec:RTS}). We obtain a uniform bound on the Fisher information associated with the truncated system \eqref{truncated system} by passing to the limit with respect to the damping parameter \(\epsilon\) and exploiting the corresponding compactness properties (see Theorem \ref{existence of solution: truncated system}). This quantity will play an important role throughout the article.

\subsection{Notaion}

\begin{itemize}
    \item [$\bullet$] $C_c^{\infty}([0,T)\times\Omega)$ denotes the space of all compactly supported smooth function in $[0,T)\times\Omega$.
    \item [$\bullet$] We denote $f^N:=(f_i^N)_{0\leq i\leq N}$ and $f_{\epsilon}^N:=(f_{i,\epsilon}^N)_{0\leq i\leq N}$, for all $N\in\mathbb{N}$.
    \item [$\bullet$] $RM([0,T)\times\Omega)$ denotes the space of all Radon measure defined on $[0,T)\times\Omega$ with absolute variance as the norm.
\end{itemize}

\section{Truncated system}\label{sec:RTS}

The exchange operators $Q_i^N$ are quadratic in $f^N$, so the right-hand side of \eqref{truncated system} is only locally Lipschitz and a priori bounds are needed before existence can be asserted. To obtain a well-posed approximating problem, we first regularize \eqref{truncated system} by damping the nonlinearity, so that for all $0\le i\le N$ and $\epsilon>0$ we consider
\begin{equation}\label{regularized truncated system}
\left \{
\begin{aligned}
    \partial_t f_{i,\epsilon}^N- d_i \Delta f_{i,\epsilon}^N= & \frac{Q_i^N(f_{\epsilon}^N)}{1+\epsilon\sum_{k=0}^N \left(f_{k,\epsilon}^N\right)^2}=:Q_{i,\epsilon}^N(f_{\epsilon}^N)\qquad && \mbox{in} \ (0,T)\times\Omega\\
    \nabla f_{i,\epsilon}^N \cdot n=& 0 \qquad && \mbox{on}\ (0,T)\times\partial\Omega\\
    f_{i,\epsilon}^N(0,x)=& f_{i,0}^N(x)\ge 0 \qquad && \mbox{in}\ \Omega.
    \end{aligned}
    \right .
\end{equation}
The damping factor renders the right-hand side globally Lipschitz. Indeed, both $Q_i^N(f_\epsilon^N)$ and its derivatives are controlled by
$\displaystyle{|Q_i^N(f_{\epsilon}^N)|\lesssim 1+\sum_{k=0}^N |f_{k,\epsilon}^N|^2}$, so the quotient $Q_{i,\epsilon}^N(f_\epsilon^N)$ is bounded with bounded gradient. By standard parabolic theory, the regularized system \eqref{regularized truncated system} therefore admits a unique smooth solution on $(0,T)\times\Omega$ (see Theorem \ref{existence theory for regular source}).

We next show that this solution is nonnegative, which we deduce from the quasipositive \eqref{quasipositive} structure of the source terms. When the $i$-th component vanishes, every loss term in $Q_{i,\epsilon}^N$ carries the factor $f_{i,\epsilon}^N$ and hence drops out, leaving only nonnegative gain terms, so $Q_{i,\epsilon}^N$ satisfies the quasipositivity condition \eqref{quasipositive}. Applying Theorem \ref{nonnegativity of solution} with $m=N+1$ (after relabeling the indices $0\le i\le N$ as $1\le i\le m$), and using $f_{i,0}^N\ge 0$, we conclude that $f_{i,\epsilon}^N\ge 0$ for all $0\le i\le N$.

Next, we derive some estimates for the solutions of the damped truncated system \eqref{regularized truncated system}. We begin with the following $L^2$ estimate.
\begin{prop}
\label{L2 estimate: truncated regularized system}
Assume that \(f_{i,0}\) and \(d_i\) satisfy \eqref{H_init} and \eqref{H_diff} respectively. For every $T>0$, there exists a constant $C(T)=C(T,\Omega,d_*,d^*)>0$, independent of $N$ and $\epsilon$, such that the solutions of the regularized system \eqref{regularized truncated system} satisfy
\begin{subequations}
\label{L2 estimate: regularized truncated system}
\begin{align}
\int_0^T \int_\Omega
\left( \sum_{i=0}^{N} i\, f_{i,\epsilon}^N(t,x) \right)^2
\,dx\,dt
&\;\le\;
C(T)\,
\frac{d^*}{d_*}
\left\|
\sum_{i=0}^{\infty} i\, f_{i,0}
\right\|_{L^2(\Omega)}^{2},
\label{L2 estimate: regularized truncated system-a}
\\
\int_0^T \int_\Omega
\left( \sum_{i=0}^{N} f_{i,\epsilon}^N(t,x) \right)^2
\,dx\,dt
&\;\le\;
C(T)\,
\frac{d^*}{d_*}
\left\|
\sum_{i=0}^{\infty} f_{i,0}
\right\|_{L^2(\Omega)}^{2},
\label{L2 estimate: regularized truncated system-b}
\\
\int_0^T \int_\Omega
\left( \sum_{i=0}^{N} (i+1)\, f_{i,\epsilon}^N(t,x) \right)^2
\,dx\,dt
&\;\le\;
C(T)\,
\frac{d^*}{d_*}
\left\|
\sum_{i=0}^{\infty} (i+1)\, f_{i,0}
\right\|_{L^2(\Omega)}^{2}.
\label{L2 estimate: regularized truncated system-c}
\end{align}
\end{subequations}
\end{prop}
The proof follows along the same lines as the $L^2$ estimates established in~\cite{desvillettes2007global, canizo2010absence}. Since the arguments are essentially the same, we omit the details.

We further note the following relation holds for the system \eqref{regularized truncated system}. The proof can be found in \cite[Lemma 1]{EV2021}.
\begin{lem}[\cite{EV2021}]\label{mass and momentum with auxiliary equ}
Let $(g_i)_{i\ge 0}$ be a sequence of non-negative real numbers. Then the solution of \eqref{regularized truncated system} satisfies
\begin{align}\label{gi technique}
\sum_{i=0}^{N} g_i \frac{d}{dt} \int_\Omega f_{i,\epsilon}^N \,dx
= \;&
\sum_{i=1}^{N} (g_{i-1} - g_i)
\int_\Omega \frac{f_{i,\epsilon}^N \left( \sum_{j=0}^{N-1} K_{i,j}\, f_{j,\epsilon}^N \right)}{1+\epsilon\sum_{k=0}^N (f_{k,\epsilon}^N)^2}\, dx \notag \\
&\quad + \sum_{i=0}^{N-1} (g_{i+1} - g_i)
\int_\Omega \frac{f_{i,\epsilon}^N \left( \sum_{j=1}^{N} K_{j,i}\, f_{j,\epsilon}^N \right)}{1+\epsilon\sum_{k=0}^N (f_{k,\epsilon}^N)^2}\, dx.
\end{align}
\end{lem}

\medskip

Note that for $g_i=1, i$, the right hand side of the above relation \eqref{gi technique} vanishes and provides conservation of the total number and mass of particles in the system. We consider the kernel 
\[
K_{i,j}:=b_ia_j.
\]
Let us define $\mathcal{Q}_j$ recursively as
\[
\mathcal{Q}_0=1, \ \ \mathcal{Q}_j= \frac{a_{j-1}}{b_{j}} \mathcal{Q}_{j-1}.
\] 
Thanks to the assumption \eqref{sectorwise bound of kernel} that $b_* (i+1)\leq b_i\leq b^*(i+1)$ and $a_j\leq a^*(j+1)^{\alpha}$, for some $a^*, b^*, b_*>0$, we have the following estimate for $\mathcal{Q}_j$:
\[
0<\mathcal{Q}_j\leq \left(\frac{a^*}{b_*}\right)^{j}(j!)^{-1+\alpha}.
\]
Using the fact that $\alpha<1$, we have that
\begin{align}\label{sum Q}
    \sum_{j=0}^{\infty} \mathcal{Q}_j<+\infty.
\end{align}
Next, we define the entropy functional
\begin{align}\label{def:entropy, RTS}
E^N(f^N_{\epsilon}):=\sum_{i=0}^N\int_{\Omega} f_{i,\epsilon}^N \log \left(\frac{f_{i,\epsilon}^N}{\mathcal{Q}_i}\right).
\end{align}
Differentiating the entropy and using the homogeneous Neumann condition, we obtain
\begin{align*}
    &\partial_t E^N(f_{\epsilon}^N)=\\
    &\partial_t \left( \int_{\Omega} f_{0,\epsilon}^N \log \left(\frac{f_{0,\epsilon}^N}{\mathcal{Q}_0}\right)\right)+ \partial_t \left(\sum_{i=1}^{N-1}\int_{\Omega} f_{i,\epsilon}^N \log \left(\frac{f_{i,\epsilon}^N}{\mathcal{Q}_i}\right)\right)+\partial_t \left( \int_{\Omega} f_{N,\epsilon}^N \log \left(\frac{f_{N,\epsilon}^N}{\mathcal{Q}_N}\right)\right)\\
    =& \int_{\Omega} \partial_tf_{0,\epsilon}^N \log \left(\frac{f_{0,\epsilon}^N}{\mathcal{Q}_0}\right)+ \sum_{i=1}^{N-1}\int_{\Omega} \partial_t f_{i,\epsilon}^N \log \left(\frac{f_{i,\epsilon}^N}{\mathcal{Q}_i}\right)+\int_{\Omega} \partial_t f_{N,\epsilon}^N \log \left(\frac{f_{N,\epsilon}^N}{\mathcal{Q}_N}\right)\\
    =& -\sum_{i=0}^N \int_{\Omega}d_i \frac{|\nabla f^N_{i,\epsilon}|^2}{f^N_{i,\epsilon}}+ \int_{\Omega} \mathcal{G}^{-1}\log \left(\frac{f_{0,\epsilon}^N}{\mathcal{Q}_0}\right) \left( f_{1,\epsilon}^N b_1 \sum_{j=0}^{N-1}a_j f_{j,\epsilon}^N-f_{0,\epsilon}^N a_0 \sum_{j=1}^N b_j f_{j,\epsilon}^N\right)\\
    +& \sum_{i=1}^{N-1}\int_{\Omega} \mathcal{G}^{-1}\log \left( \frac{f_{i,\epsilon}^N}{\mathcal{Q}_i}\right) \left( f_{i+1,\epsilon}^Nb_{i+1} \sum_{j=0}^{N-1}a_j f_{j,\epsilon}^N-f_{i,\epsilon}^N b_i \sum_{j=0}^{N-1} a_jf_{j,\epsilon}^N\right)\\
    +& \sum_{i=1}^{N-1}\int_{\Omega} \mathcal{G}^{-1}\log \left( \frac{f_{i,\epsilon}^N}{\mathcal{Q}_i}\right)\left(-f_{i,\epsilon}^N a_i \sum_{j=1}^N b_j f_{j,\epsilon}^N+f_{i-1,\epsilon}^Na_{i-1} \sum_{j=1}^N b_jf_{j,\epsilon}^N\right)\\
    +& \int_{\Omega} \mathcal{G}^{-1}\log \left(\frac{f_{N,\epsilon}^N}{\mathcal{Q}_N}\right) \left( -f_{N,\epsilon}^N b_N \sum_{j=0}^{N-1}a_j f_{j,\epsilon}^N+ f_{N-1,\epsilon}^N a_{N-1} \sum_{j=1}^N b_j f_{j,\epsilon}^N\right),
\end{align*}
where $\displaystyle{\mathcal{G}:=1+\epsilon\sum_{k=0}^N\left(f_{k,\epsilon}^N\right)^2}$. Define $\displaystyle{A^N_{\epsilon}= \sum_{j=0}^{N-1} a_jf_{j,\epsilon}^N}$ and $\displaystyle{B^N_{\epsilon}=\sum_{j=1}^{N}b_jf_{j,\epsilon}^N}$. Furthermore define
\begin{equation*}
    I_{j,\epsilon}^N=\left\{
    \begin{aligned}
        & a_jf_{j,\epsilon}^N B_{\epsilon}^N-b_{j+1}f_{j+1,\epsilon}^N A_{\epsilon}^N, \ \ && 0\leq j\leq N-1\\
        & 0 \ \ && \text{otherwise}.
    \end{aligned}
    \right .
\end{equation*}
The time derivative of entropy functional can be rewritten as
\begin{align}\label{entropy dissipation functional RTS}
\partial_t E^N(f_{\epsilon}^N)= -\sum_{i=0}^N \int_{\Omega} d_i\frac{|\nabla f^N_{i,\epsilon}|^2}{f^N_{i,\epsilon}}+ \int_{\Omega}\sum_{j=0}^N \mathcal{G}^{-1} (I_{j-1,\epsilon}^N-I_{j,\epsilon}^N) \log \left(\frac{f_{j,\epsilon}^N}{\mathcal{Q}_j}\right).
\end{align}
Thanks to \cite[Lemma 5]{EV2021}, we conclude that
\[
D_N(f^N_{\epsilon}):= -\int_{\Omega}\mathcal{G}^{-1}\sum_{j=0}^N (I_{j-1,\epsilon}^N-I_{j,\epsilon}^N) \log \left(\frac{f_{j,\epsilon}^N}{\mathcal{Q}_j}\right)\geq 0.
\]
The entropy-entropy dissipation inequality can be expressed as
\begin{align}\label{entropy diss ineq RTS}
     E^N(f_{\epsilon}^N)(t)+\sum_{i=0}^N \int_0^t\int_{\Omega} d_i\frac{|\nabla f_{i,\epsilon}^N|^2}{f^N_{i,\epsilon}}+ \int_0^t D_N =  E^N(f^N_{\epsilon})(0).
\end{align}
We denote the second quantity as Fisher information, i.e.,
\begin{align}\label{Fisher info RTS}
    \mathcal{F}^N(f_{\epsilon}^N):= \sum_{i=0}^N \int_0^t\int_{\Omega} \frac{|\nabla f^N_{i,\epsilon}|^2}{f^N_{i,\epsilon}}.
\end{align}
We intend to show that the Fisher information is uniformly bounded (uniform in $N$ and $\epsilon$). Note that the quantity is nonnegative. We start analyzing the entropy functional. It can be rewritten as:
\begin{align*}
E^N(f_{\epsilon}^N)=&\sum_{i=0}^N\left(\int_{\{x\in\Omega: \ f_{i,\epsilon}^N(x)<\mathcal{Q}_i\}} f_{i,\epsilon}^N \log \left(\frac{f_{i,\epsilon}^N}{\mathcal{Q}_i}\right)+\int_{\{x\in\Omega: \ f_{i,\epsilon}^N(x)\geq \mathcal{Q}_i\}} f_{i,\epsilon}^N \log \left(\frac{f_{i,\epsilon}^N}{\mathcal{Q}_i}\right)\right)\\
&:=E_1^N(f_{\epsilon}^N)+E_2^N(f_{\epsilon}^N).
\end{align*}
The second quantity $\displaystyle{E_2^N:=\sum_{i=0}^N \int_{\{x\in\Omega: \ f_{i,\epsilon}^N(x)\geq \mathcal{Q}_i\}} f_{i,\epsilon}^N \log \left(\frac{f_{i,\epsilon}^N}{\mathcal{Q}_i}\right)}$ is nonnegative. Thanks to the fact that $\min_{\{x\geq 0\}}x\log \left(\frac{x}{\mathcal{Q}_i}\right)=-e^{-1}\mathcal{Q}_i$, and \eqref{entropy diss ineq RTS}, we obtain
\[
-e^{-1}\sum_{i=0}^N \mathcal{Q}_i\mathcal{L}\{x\in\Omega: \ f_{i,\epsilon}^N(x)<\mathcal{Q}_i \} \leq E_1^N(f_{\epsilon}^N),
\]
where $\mathcal{L}$ stands for the Lebesgue measure. The following estimate holds
\[
-\infty <-e^{-1}|\Omega|\sum_{i=0}^{\infty} \mathcal{Q}_i \leq -e^{-1}\sum_{i=0}^N \mathcal{Q}_i\mathcal{L}\{x\in\Omega: \ f_{i,\epsilon}^N(x)<\mathcal{Q}_i \} \leq E_1^N(f_{\epsilon}^N)\leq E^N(f_{\epsilon}^N)(0),
\]
where the uniform bound on the l.h.s. follows from the non-negativity of the solution and \eqref{sum Q} and the uniform bound on r.h.s follows from \eqref{def:entropy in}. From the above analysis, we conclude that 
\begin{align}\label{uniform bound on fisher info RTS}
    \mathcal{F}^N(f_{\epsilon}^N):= \sum_{i=0}^N \int_0^t\int_{\Omega} \frac{|\nabla f^N_{i,\epsilon}|^2}{f^N_{i,\epsilon}} \leq M_{\mathcal{F}}<+\infty,
\end{align}
where $M_{\mathcal{F}}$ is a positive constant, uniform with the index `$N$' and `$\epsilon$'.

It remains to remove the damping. We have the following theorem.
\begin{thm}\label{existence of solution: truncated system}
For every $0 \le i < N$, there exists a nonnegative function 
\[
f_i^N \in L^1\bigl((0,T);W^{1,1}(\Omega)\bigr),
\]
which is a nonnegative weak global-in-time solution to the truncated system \eqref{truncated system}. Moreover, $f_i^N$ is obtained as the limit of $f_{i,\varepsilon}^N$ in $L^1\bigl((0,T);W^{1,1}(\Omega)\bigr)$, where $f_{i,\varepsilon}^N$ solves the regularized truncated system
\eqref{regularized truncated system}. Furthermore, we have following estimates 

\begin{subequations}
\label{L2 estimate: truncated system}
\begin{align}
\int_{\Omega} \left( \sum_{i=0}^N (i+1) f_i^N(t,x)\right)\, dx
&\leq
\left\| \sum_{i=0}^N (i+1) f_{i,0}\right\|_{L^1(\Omega)},
\qquad \forall\, t\geq 0,
\label{L2 estimate: truncated system-a}
\\
\int_0^T \int_\Omega
\left( \sum_{i=0}^{N} i f_i^N(t,x) \right)^2
\,dx\,dt
&\;\lesssim\;
 \frac{d^*}{d_*}
\left\|
\sum_{i=0}^{\infty} i f_{i,0}
\right\|_{L^2(\Omega)}^{2},
\label{L2 estimate: truncated system-b}
\\
\int_0^T \int_\Omega
\left( \sum_{i=0}^{N} f_i^N(t,x) \right)^2
\,dx\,dt
&\;\lesssim\;
 \frac{d^*}{d_*} 
\left\|
\sum_{i=0}^{\infty} f_{i,0}
\right\|_{L^2(\Omega)}^{2},
\label{L2 estimate: truncated system-c}
\\
\int_0^T \int_\Omega
\left( \sum_{i=0}^{N} (i+1) f_i^N(t,x) \right)^2
\,dx\,dt
&\;\lesssim\;
 \frac{d^*}{d_*} 
\left\|
\sum_{i=0}^{\infty} (i+1) f_{i,0}
\right\|_{L^2(\Omega)}^{2},
\label{L2 estimate: truncated system-d}
\\
\sum_{i=0}^N \int_0^T \int_{\Omega}
\frac{|\nabla f_i^N|^2}{f_i^N}
&\leq M_{\mathcal{F}}.
\label{L2 estimate: truncated system-e}
\end{align}
\end{subequations}
where $M_{\mathcal{F}}$ is a positive constant independent of $N$ and $\epsilon$ and `$T$'.
\end{thm}
\begin{proof}
    Proof follows form \cite{Pierre2010}, Proposition \ref{L2 estimate: truncated regularized system}, relation \eqref{uniform bound on fisher info RTS} and Fatou's lemma.
\end{proof}
We want to further capture properties of the solution $f_i^N$ through sublevel sets. To this end, we introduce the following function, which can be viewed as the truncation of the identity function. The existence of such a function can be found in \cite{JF2015}.
\subsection{Truncation to identity}\label{truncation to identity multi D} Let $\Lambda>0$. For all $0\leq i\leq N$, we consider the functions $\phi_i^{\Lambda}: (\mathbb{R}_{\geq 0})^{N+1}\to\mathbb{R}_{\geq 0}\to$, satisfies
\begin{itemize}
    \item [$\bullet$] $\phi_i^{\Lambda}\in C^2((\mathbb{R}_{\geq 0})^{N+1})$.
    \item[$\bullet$] There exists $K_1>0$, such that $\displaystyle{\sqrt{v_j}\sqrt{v_k}\left| \partial_j \partial_k (\phi_i^{\Lambda}(v)\right|\leq K_1}$, \ for all $v\in(\mathbb{R}_{\geq 0})^{N+1}$, $j,k\in\{0,\cdots,N\}$ and for all $\Lambda$.
    \item[$\bullet$] For each $\Lambda>0$, $\mathrm{supp}\left\{D(\phi^{\Lambda}_i)\right\}$ is compact, say  $\mathrm{supp}\left\{D(\phi^{\Lambda}_i)\right\} \subset B(0,\Lambda^*)$. 
    \item[$\bullet$] For all $v\in(\mathbb{R}_{\geq 0})^{N+1}$, $\displaystyle{\lim_{\Lambda\to\infty} \partial_j \phi_i^{\Lambda}(v)=\delta_{ij}}$.
    \item[$\bullet$] There exists $K_2> 0$, such that $\displaystyle{\left| \partial_j \phi_i^{\Lambda}(v)\right|\leq K_2, \ \forall\, v\in(\mathbb{R}_{\geq 0})^{N+1},  \ \forall\, \Lambda>0}$.
    \item[$\bullet$] $\phi^{\Lambda}_i(v)=v_i$, \ $\forall\, v\in(\mathbb{R}_{\geq 0)})^{N+1}$, where $\sum_{j=0}^N v_j\leq \Lambda$.
    \item[$\bullet$] For all $K_3>0$, $\sum_{i=0}^N v_j>K_3$, implies $\sum_{i=0}^N \phi_i^{\Lambda}(v)\geq \min\{K_3,\Lambda\}$, for all $\Lambda>0$ and $v\in(\mathbb{R}_{\geq 0})^{N+1}$.
    \item[$\bullet$] $\displaystyle{\lim_{\Lambda\to\infty} \sup_{|v|\leq K}\left| \partial_j \partial_k \phi_i^{\Lambda}(v)\right|=0, \ \forall \, K>0}$. 
\end{itemize}
\begin{rem}\label{existence of such truncation to identity}
  A suitable function with the required properties can be constructed as in \cite{JF2015}. Let \(\mathcal{X}\in C^{\infty}(\mathbb{R})\) be a smooth, nonincreasing function taking values in \([0,1]\), such that
\[
\mathcal{X} \equiv 1 \quad \text{on } (-\infty,0),
\qquad
\mathcal{X} \equiv 0 \quad \text{on } (1,\infty).
\]
For \(\Lambda>0\), define
\[
\phi_i^{\Lambda}(v)
:=
v_i\,\mathcal{X}\!\left(\frac{\sum_{k=1}^{S} v_k - \Lambda}{\Lambda}\right)
+
3\Lambda\left(
1-\mathcal{X}\!\left(\frac{\sum_{k=1}^{S} v_k - \Lambda}{\Lambda}\right)
\right).
\]
One can verify that this function satisfies all the properties stated in \ref{truncation to identity multi D}. For further details, we refer the reader to \cite{JF2015}.
\end{rem}
Let us now look at the damped truncated system \eqref{regularized truncated system}. For all $\xi\in C_c^{\infty}([0,T)\times\Omega)$, the following identity holds:
\begin{align}
    \int_0^T \int_{\Omega} \partial_t \left(\phi_i^{\Lambda}(f_{\epsilon}^N)\right) \xi \, dx\, dt=& \int_0^T \int_{\Omega} \sum_{j=0}^N \partial_j \phi_i^{\Lambda}\left(f_{\epsilon}^N\right) \frac{\partial f_{j,\epsilon}^N}{\partial t} \xi \, dx\, dt \label{RTS, time derivative through truncation}\\
    \underset{\rm{integration\  by \ parts}}{=}-\int_{\Omega} \phi^{\Lambda}_i \left(f_{0}^N\right) \xi(0)\, dx -&\int_{0}^T\int_{\Omega} \phi^{\Lambda}_i \left(f_{\epsilon}^N \right) \partial_t\xi \, dx\, dt, \nonumber
\end{align}
where $f_0^N:=(f_{0,0},\cdots f_{N,0})$. We analyze the second term of the above expression further:
\begin{align*}
    \int_0^T \int_{\Omega} \sum_{j=0}^N &\partial_j \phi_i^{\Lambda}\left(f_{\epsilon}^N\right) \frac{\partial f_{j,\epsilon}^N}{\partial t} \xi \, dx\, dt= \sum_{j=0}^N \int_0^T\int_{\Omega} \left(d_j \Delta f_{j,\epsilon}^N+Q_{j,\epsilon}^N\right) \partial_j \phi_i^{\Lambda}\left(f_{\epsilon}^N\right)  \xi \, dx\, dt.\\
    =& \sum_{j=0}^N \int_0^T \int_{\Omega} \partial_j \phi_i^{\Lambda}\left(f_{\epsilon}^N\right) Q^N_{j,\epsilon}  \xi \, dx\, dt-\sum_{j=0}^N \int_0^T \int_{\Omega} d_j \partial_j \phi_i^{\Lambda}\left(f_{\epsilon}^N\right) \nabla f_{j,\epsilon}^N \cdot \nabla \xi \, dx\, dt\\
    & \ \ \ \ \ -\sum_{j,k=0}^N \int_0^T \int_{\Omega} d_j \partial_{jk} \phi_i^{\Lambda}\left(f_{\epsilon}^N\right) \nabla f_{j,\epsilon}^N \cdot\nabla f_{k,\epsilon}^N \xi \, dx\, dt.
\end{align*}
Combining the above estimate with \eqref{RTS, time derivative through truncation}, we obtain
\begin{align}
-\int_{\Omega} &\phi^{\Lambda}_i \left(f_{0}^N\right) \xi(0)\, dx -\int_{0}^T\int_{\Omega} \phi^{\Lambda}_i \left(f_{\epsilon}^N \right) \partial_t\xi \, dx\, dt \label{RTS: composition with the smooth function}\\
=& \sum_{j=0}^N \int_0^T \int_{\Omega} \partial_j \phi_i^{\Lambda}\left(f_{\epsilon}^N\right) Q^N_{j,\epsilon}  \xi \, dx\, dt- \sum_{j=0}^N \int_0^T \int_{\Omega} d_j \partial_j \phi_i^{\Lambda}\left(f_{\epsilon}^N\right) \nabla f_{j,\epsilon}^N \cdot \nabla \xi \, dx\, dt \nonumber\\
    & \ \ \ \ \ -\sum_{j,k=0}^N \int_0^T \int_{\Omega} d_j \partial_{jk} \phi_i^{\Lambda}\left(f_{\epsilon}^N\right) \nabla f_{j,\epsilon}^N \cdot\nabla f_{k,\epsilon}^N \xi \, dx\, dt.\nonumber    
\end{align}
Here, the last integral defined through the following estimate
\begin{align}
    d_j \int_{\{| f_{j,\epsilon}^N|< \Lambda^*\}} \left| \nabla f_{j,\epsilon}^N\right|^2 \, dx\, dt \leq \Lambda^* \left( \int_0^T\int_{\Omega} | Q_{j,\epsilon}^N| \, dx\, dt +\int_{\Omega} | f_{j,0}|\, dx \right),
\end{align}
where $\Lambda^*$ defined as the radius of the compact support of the derivative of $\phi_{i}^{\Lambda}$, as described in \ref{truncation to identity multi D}. The above estimate is the consequence of the following lemma.
\begin{lem}[Truncation energy estimate]\label{lem:trunc-energy}
Let $d>0$, let $\Theta \in L^{1}(Q_T)$, and let $\mathcal{F}_{\rm{in}} \in L^{1}(\Omega)$.
Let $\mathcal{F}$ be the solution of the system
\begin{equation}\label{eq:trunc-heat}
\begin{cases}
\partial_t \mathcal{F} - d \Delta \mathcal{F} = \Theta & \text{in } (0,T)\times \Omega \\[2mm]
\nabla \mathcal{F} \cdot n = 0, & \text{on } (0,T)\times\partial\Omega \\[2mm]
\mathcal{F}(0,\cdot) = \mathcal{F}_{\rm{in}}, & \text{in } \Omega.
\end{cases}
\end{equation}
Then, for every $M>0$, the following estimate holds:
\begin{equation}\label{eq:trunc-energy}
d \int_{\{|\mathcal{F}|\le M\}} |\nabla \mathcal{F}|^2
\le
M\left(
\int_{(0,T)\times\Omega} |\Theta|
+
\int_{\Omega} |\mathcal{F}_{\rm{in}}|
\right).
\end{equation}
\end{lem}
The proof of the lemma can be found in \cite[Lemma 5.7]{Pierre2010}. The limiting asymptotics $\epsilon\to 0$ in \eqref{RTS: composition with the smooth function}, yields the following Proposition.
\begin{prop}\label{TS: composition with the truncated smooth function}
    For all $0\leq i\leq N$, let $f_i^N$ be the $L^1((0,T);W^{1,1}(\Omega))$ solution to \eqref{truncated system}, obtained as the limit of $f_{i,\varepsilon}^N$ in $L^1\bigl((0,T);W^{1,1}(\Omega)\bigr)$, where $f_{i,\varepsilon}^N$ solves the regularized truncated system
\eqref{regularized truncated system}. Let $\phi_i^{\Lambda}$ be the function defined in \ref{truncation to identity multi D}. Then the following holds.
\begin{align}
 -\int_{\Omega} &\phi^{\Lambda}_i \left(f_{0}^N\right) \xi(0)\, dx -\int_{0}^T\int_{\Omega} \phi^{\Lambda}_i \left(f^N \right) \partial_t\xi \, dx\, dt   \label{TS: composition with the truncation of identity}\\
=& \sum_{j=0}^N \int_0^T \int_{\Omega} \partial_j \phi_i^{\Lambda}\left(f^N\right) Q^N_{j}  \xi \, dx\, dt- \sum_{j=0}^N \int_0^T \int_{\Omega} d_j \partial_j \phi_i^{\Lambda}\left(f^N\right) \nabla f_{j}^N \cdot \nabla \xi \, dx\, dt \nonumber\\
    & \ \ \ \ \ + \int_0^T \int_{\Omega} \xi \, d\mu_i^{\Lambda}(x,t),\nonumber    
\end{align}
where $\mu^{\Lambda}_i$ denotes a sequence of Radon measure satisfying
\begin{align}\label{TS: measure seq goes to zero}
    \lim_{\Lambda\to \infty}\left| \mu_i^{\Lambda}\right| ([0,T)\times\Omega)=0,
\end{align}
for all $T>0$ and all $i$. 
\end{prop}
\begin{proof}
We use the relation \eqref{RTS: composition with the smooth function}
\begin{align*}
-\int_{\Omega} &\phi^{\Lambda}_i \left(f_{0}^N\right) \xi(0)\, dx -\int_{0}^T\int_{\Omega} \phi^{\Lambda}_i \left(f_{\epsilon}^N \right) \partial_t\xi \, dx\, dt:=-I_1-I_2 \\
=& \sum_{j=0}^N \int_0^T \int_{\Omega} \partial_j \phi_i^{\Lambda}\left(f_{\epsilon}^N\right) Q^N_{j,\epsilon}  \xi \, dx\, dt- \sum_{j=0}^N \int_0^T \int_{\Omega} d_j \partial_j \phi_i^{\Lambda}\left(f_{\epsilon}^N\right) \nabla f_{j,\epsilon}^N \cdot \nabla \xi \, dx\, dt \nonumber\\
    & \ \ \ \ \ -\sum_{j,k=0}^N \int_0^T \int_{\Omega} d_j \partial_{jk} \phi_i^{\Lambda}\left(f_{\epsilon}^N\right) \nabla f_{j,\epsilon}^N \cdot\nabla f_{k,\epsilon}^N \xi \, dx\, dt:=I_3+I_4+I_5\nonumber    
\end{align*}
We pass to the limit $\epsilon\to 0$ in each of the terms. Note that $I_1$ is independent of $\epsilon$. Hence we will consider rest of the four terms only. Consider the second term:
\begin{align*}
    &\left| \int_0^T \int_{\Omega} \phi_i^{\Lambda}(f_{\epsilon}^N) \partial_t \xi \, dx\, dt- \int_0^T \int_{\Omega} \phi_i^{\Lambda}(f^N) \partial_t \xi \, dx\, dt\right| \\
    & \leq \sup |\partial_t \xi| \int_0^T \int_{\Omega} \left| \phi^{\Lambda}_i \left(f_{\epsilon}^N\right)-\phi^{\Lambda}_i(f^N)\right| \, dx\, dt\\
    & \leq \sup |\partial_t \xi| \int_0^T \int_{\Omega}  \sup\left\{\left|\partial_j\phi_i^{\Lambda}(\cdot)\right|\right\} \sum_{j=1}^N \left| f_{j,\epsilon}^N-f_j^N\right| \, dx\, dt \underset{\epsilon \to 0}{\longrightarrow} 0.
\end{align*}
It yields
\begin{align}\label{TS:limit of I2}
\lim_{\epsilon\to 0} I_2= \int_0^T \int_{\Omega} \phi^{\Lambda}_i (f^N) \partial_t\xi\, dx\, dt.
\end{align}
The passage to the limit $\epsilon \to 0$ in $I_3$ is a straightforward application of dominated convergence theorem. It yields
\begin{align}\label{TS:limit of I3}
\lim_{\epsilon\to 0} I_3= \sum_{j=0}^N\int_0^T \int_{\Omega} \partial_j\phi^{\Lambda}_i (f^N) Q_{j}^N\xi\, dx\, dt.
\end{align}
We now look towards the term $I_4$. Consider the following calculation.
\begin{align*}
    &\left| \sum_{j=0}^N \int_0^T \int_{\Omega} d_j \partial_j \phi^{\Lambda}\left(f_{\epsilon}^N\right) \nabla f_{j,\epsilon}^N \cdot \nabla \xi \, dx\, dt- \sum_{j=0}^N \int_0^T \int_{\Omega} d_j \partial_j \phi^{\Lambda}\left(f^N\right) \nabla f_{j}^N \cdot \nabla \xi \, dx\, dt \right|\\
    \leq & \left| \sum_{j=0}^N \int_0^T \int_{\Omega} d_j \partial_j \phi^{\Lambda}\left(f_{\epsilon}^N\right) \nabla f_{j,\epsilon}^N \cdot \nabla \xi \, dx\, dt- \sum_{j=0}^N \int_0^T \int_{\Omega} d_j \partial_j \phi^{\Lambda}\left(f_{\epsilon}^N\right) \nabla f_{j}^N \cdot \nabla \xi \, dx\, dt \right|\\
    +& \left| \sum_{j=0}^N \int_0^T \int_{\Omega} d_j \partial_j \phi^{\Lambda}\left(f_{\epsilon}^N\right) \nabla f_{j}^N \cdot \nabla \xi \, dx\, dt- \sum_{j=0}^N \int_0^T \int_{\Omega} d_j \partial_j \phi^{\Lambda}\left(f^N\right) \nabla f_{j}^N \cdot \nabla \xi \, dx\, dt \right|.
\end{align*}
Using the dominated convergence theorem, we have that the above term vanishes as $\epsilon \to 0$. It yields
\begin{align}\label{TS:limit of I4}
 \lim_{\epsilon\to 0} I_4= -\sum_{j=0}^N \int_0^T \int_{\Omega} d_j \partial_j \phi\left(f^N\right) \nabla f_j^N \cdot\nabla \xi \, dx\, dt.   
\end{align}
The last term $I_5$, can be interpreted in the following way
\begin{align*}
    I_5=-\sum_{j,k=0}^N \int_0^T \int_{\Omega} d_j \partial_{jk} \phi_i^{\Lambda}\left(f_{\epsilon}^N\right) \nabla f_{j,\epsilon}^N \cdot\nabla f_{k,\epsilon}^N \xi \, dx\, dt= \int_{0}^T\int_{\Omega} \xi d\mu_{i,\epsilon}^{\Lambda}(x,t),
\end{align*}
where the radon measure $\mu_{i,\epsilon}^{\Lambda}$ is defined as
\begin{align}
\mu_{i,\epsilon}^{\Lambda}:=& -\sum_{j,k=0}^N  d_j \partial_{jk} \phi_i^{\Lambda}\left(f_{\epsilon}^N\right) \nabla f_{j,\epsilon}^N \cdot\nabla f_{k,\epsilon}^N \xi \, dx\, dt \label{RTS: intermediate measure}\\
=& -\sum_{j,k=0}^N  4d_j \sqrt{f_{j,\epsilon}} \sqrt{f_{k,\epsilon}^N}\partial_{jk} \phi_i^{\Lambda}\left(f_{\epsilon}^N\right) \nabla \sqrt{\left(f_{j,\epsilon}^N\right)} \cdot\sqrt{\left(\nabla f_{k,\epsilon}^N\right)} \xi \, dx\, dt.\nonumber
\end{align}
Using \ref{truncation to identity multi D} and Theorem \ref{existence of solution: truncated system}, we obtain
\[
|\mu_{i,\epsilon}^{\Lambda}|([0,T)\times\Omega) \leq 8K_2\sup_{j\in\mathbb{N}\cup\{0\}} \{d_j\} \sum_{j=0}^N \int_0^T\int_{\Omega} \left| \nabla\sqrt{\left( f_{j,\epsilon}^N\right)}\right|^2\, dx\, dt\leq 8K_2\sup_{j\in\mathbb{N}\cup\{0\}} \{d_j\} M_{\mathcal{F}}.
\]
Hence we conclude that $\mu^{\Lambda}_{i,\epsilon}\overset{*}{\rightharpoonup} \mu^{\Lambda}_i$ for all $i=0,\cdots,N$. It yields
\begin{align}\label{TS:limit of I5}
    \lim_{\epsilon\to 0} I_5= \int_0^T\int_{\Omega} \xi d\mu_{i}^{\Lambda}(x,t).
\end{align}
Combining \eqref{TS:limit of I2}, \eqref{TS:limit of I3}, \eqref{TS:limit of I4} and \eqref{TS:limit of I5}, we obtain \eqref{TS: composition with the truncation of identity}. Next we consider the following measures
\[
\nu_{j,\epsilon}^K:= \chi_{\left\{\left|f_{j,\epsilon}^N\right|\in[K-1,K)\right\}}\left| \nabla \sqrt{\left(f_{j,\epsilon}^N\right)}\right|^2 \, dx\, dt,
\]
on $[0,T)\times\Omega$. Using \eqref{RTS: intermediate measure}, we calculate the total variation of $\mu_{i,\epsilon}^N$. The following computation holds
\begin{align}
 |\mu_{i,\epsilon}^{\Lambda}|&([0,T)\times\Omega) \label{RTS: second intermediate measure}\\ 
 &=\sum_{j,k=0}^N \int_0^T\int_{\Omega}  4d_j \left|\sqrt{f_{j,\epsilon}^N} \sqrt{f_{k,\epsilon}^N}\partial_{jk} \phi_i^{\Lambda}\left(f_{\epsilon}^N\right)\right| \left| \nabla \sqrt{\left(f_{j,\epsilon}^N\right)} \cdot\sqrt{\left(\nabla f_{k,\epsilon}^N\right)} \right| \, dx\, dt   \nonumber\\
 &\leq 8(N+1)\sup_{j\geq 0}\{d_j\} \times \nonumber \\
 &\sum_{j=0}^N \sum_{K=1}^{\infty}\int_0^T\int_{\Omega}   \chi_{\left\{\left|f_{j,\epsilon}^N\right|\in[K-1,K)\right\}}\left|\sqrt{f_{j,\epsilon}^N} \sqrt{f_{k,\epsilon}^N}\partial_{jk} \phi_i^{\Lambda}\left(f_{\epsilon}^N\right)\right| \left|\nabla \sqrt{\left(f_{j,\epsilon}^N\right)} \right|^2  \, dx\, dt   \nonumber\\
 & \leq 8(N+1)\sup_{j\geq 0}\{d_j\} \sum_{j=0}^N\sum_{K=1}^{\infty} \nu_{j,\epsilon}^K ([0,T)\times\Omega)\cdot \sup_{|v|\in[K-1,K); 0\leq j,k\leq N} \sqrt{v_j}\sqrt{v_k}\left| \partial_j\partial_k \phi_i^{\Lambda}(v)\right|. \nonumber
\end{align}
Note that, thanks to \ref{truncation to identity multi D}, only finitely many terms survive in the sum. Let till $K=K_{max}<+\infty$ the sum survives. Furthermore, thanks to Theorem \ref{existence of solution: truncated system}, we have that 
\[
\sum_{K=1}^{K_{max}} \nu_{j,\epsilon}^K([0,T)\times\Omega)\leq \int_0^T\int_{\Omega} \left| \nabla \sqrt{\left(f_{j,\epsilon}^N\right)}\right|^2\, dx\, dt \leq M_{\mathcal{F}}.
\]
This further yields
\begin{align} \label{RTS: third intermediate measure}
    M_{\mathcal{F}}  \geq \limsup_{\epsilon\to 0}\sum_{K=1}^{K_{max}} \nu_{j,\epsilon}^K([0,T)\times\Omega).
\end{align}
The boundednes and nonnegativity of $\nu_{j,\epsilon}^N$ implies that there exists a subsequence (we still index it by $\epsilon$)  such that $\displaystyle{\lim_{\epsilon\to 0}\nu_{j,\epsilon}^K([0,T)\times\Omega)}$ exists. Hence the above relation \eqref{RTS: third intermediate measure} can be written as
\begin{align} \label{RTS: third intermediate measure new 1}
    M_{\mathcal{F}}  \geq \sum_{K=1}^{K_{max}} \lim_{\epsilon\to 0}\nu_{j,\epsilon}^K([0,T)\times\Omega).
\end{align}
Using the fact that measure of open sets is lower semicontinuous with respect to weak-* convergence and \eqref{RTS: second intermediate measure}, we obtain the following
\begin{align*}
    \left| \mu_{i}^{\Lambda}\right|  ([0,T)\times\Omega)\leq \liminf_{\epsilon\to 0} \left| \mu_{i}^{\Lambda} \right|([0,T)\times\Omega)
     \end{align*}
             \begin{align*}
    &  \leq 8(N+1)\sup_{j\geq 0}\{d_j\} \sum_{j=0}^N\lim_{\epsilon\to 0}\sum_{K=1}^{\infty} \nu_{j,\epsilon}^K ([0,T)\times\Omega)\cdot \sup_{|v|\in[K-1,K); 0\leq j,k\leq N} \sqrt{v_j}\sqrt{v_k}\left| \partial_j\partial_k \phi_i^{\Lambda}(v)\right|\\
    &\leq 8(N+1)\sup_{j\geq 0}\{d_j\} \sum_{j=0}^N\sum_{K=1}^{K_{max}} \lim_{\epsilon\to 0}\nu_{j,\epsilon}^K ([0,T)\times\Omega)\cdot \sup_{|v|\in[K-1,K); 0\leq j,k\leq N} \sqrt{v_j}\sqrt{v_k}\left| \partial_j\partial_k \phi_i^{\Lambda}(v)\right|\\
    & \leq  8(N+1)\sup_{j\geq 0}\{d_j\} \sum_{j=0}^N\sum_{K=1}^{\infty}\liminf_{\epsilon\to 0} \nu_{j,\epsilon}^K ([0,T)\times\Omega)\cdot \sup_{|v|\in[K-1,K); 0\leq j,k\leq N} \sqrt{v_j}\sqrt{v_k}\left| \partial_j\partial_k \phi_i^{\Lambda}(v)\right|
\end{align*}
where we use \eqref{RTS: third intermediate measure new 1} in the third line. We have that
\[
\left | \liminf_{\epsilon\to 0} \nu_{j,\epsilon}^K ([0,T)\times\Omega)\cdot \sup_{|v|\in[K-1,K); 0\leq j,k\leq N} \sqrt{v_j}\sqrt{v_k}\left| \partial_j\partial_k \phi_i^{\Lambda}(v)\right| \right| \leq K_1 \liminf_{\epsilon\to 0} \nu_{j,\epsilon}^K ([0,T)\times\Omega),
\]
where 
\[
\sum_{K=1}^{\infty} \liminf_{\epsilon\to 0} \nu_{j,\epsilon}^K ([0,T)\times\Omega) \leq \liminf_{\epsilon\to 0}\sum_{K=1}^{\infty} \nu_{j,\epsilon}^K ([0,T)\times\Omega) \leq M_{\mathcal{F}}
\]
thanks to Fatou's lemma. Hence using dominated convergence theorem on counting measure, we obtain that
\begin{align*}
&\lim_{\Lambda\to\infty}\left| \mu_{i}^{\Lambda}\right| ([0,T)\times\Omega) \\
&\leq  8(N+1)\sup_{j\geq 0}\{d_j\} \times\\
&\quad \qquad  \sum_{j=0}^N \sum_{K=1}^{\infty} \liminf_{\epsilon\to 0} \nu_{j,\epsilon}^K ([0,T)\times\Omega)\cdot \lim_{\Lambda\to\infty}\sup_{|v|\in[K-1,K); 0\leq j,k\leq N} \sqrt{v_j}\sqrt{v_k}\left| \partial_j\partial_k \phi_i^{\Lambda}(v)\right|\\
&=0.
\end{align*}
This concludes the proof of the proposition.
\end{proof}
The above theorem helps us to show that $f^N:=(f_0^N,\cdots,f_N^N)$, obtained as the limit of $f_{\epsilon}^N:=(f_{0,\epsilon}^N,\cdots f_{N,\epsilon}^N)$ is an renormalize solution of the truncated system \eqref{truncated system}. To establish this we need the following lemma. The proof of the lemma can be found in \cite[Lemma 4]{JF2015}.
\begin{lem}\label{TS: auxiliary lemma, renormalized solution}
    Let $\Omega$ be a bounded domain with Lipschitz boundary. Let $T>0$ and $v\in L^1\left((0,T); \left(L^1(\Omega)\right)^{N+1}\right)\cap L^2\left((0,T); \left(H^1(\Omega)\right)^{N+1}\right)$. Let $v_0\in \left(L^1(\Omega)\right)^{N+1}$, $\nu_i\in RM([0,T)\times\Omega)$, $w_i\in L^1((0,T);L^1(\Omega))$ and $z_i\in L^2\left([0,T); \left(L^2(\Omega)\right)^d\right)$. Suppose that for all $i=0,\cdots N$ and for any $\xi\in C_c^{\infty}([0,T)\times\Omega)$, the following relation holds
    \begin{align}
        -\int_0^T\int_{\Omega} v_i\partial_t \xi \, dx\, dt- \int_{\Omega} (v_0)_i\xi(0) \, dx= &\int_0^T\int_{\Omega} z_i \nabla\xi \, dx\, dt \label{towards renormalize solution, intermediate 1}\\
        &+\int_0^T\int_{\Omega} w_i\xi \, dx\, dt+ \int_0^T\int_{\Omega} \xi \, d\nu_i. \nonumber
    \end{align}
    Then for all smooth function $\psi:\mathbb{R}^{N+1}\to \mathbb{R}$ with compactly supported first derivatives, the following holds
    \begin{align*}
        \Bigg| &-\int_0^T\int_{\Omega} \psi(v) \partial_t \xi \, dx\, dt- \int_{\Omega} \psi(v_0)\xi(0)\, dx- \sum_{i=0}^N \int_0^T\int_{\Omega} \xi \partial_i \psi(v) w_i\, dx\, dt\\
        &-\sum_{i=0}^N \int_0^T\int_{\Omega} \partial_i \psi(v) z_i\cdot\nabla\xi \, dx\, dt-\sum_{i=0}^N\sum_{k=0}^N \int_0^T\int_{\Omega} \xi \partial_{i}\partial_{k}\psi(v) z_i\cdot\nabla v_k\, dx\, dt\Bigg|\\
        &\leq C(\Omega) \| \xi\|_{L^{\infty}([0,T)\times\Omega)} \left( \sup_{x\in \mathbb{R}^{N+1}}|D\psi(x)|\right) \sum_{i=0}^N \| \nu_i \|_{RM([0,T)\times\Omega)},
    \end{align*}
    where $C(\Omega)$ is some positive constant depending on the domain.
\end{lem} 
Now we are well equipped to show that, for $i=0,\cdots, N$, the solution $f_i^N$ to \eqref{truncated system} as defined in Theorem \eqref{existence of solution: truncated system}, is indeed a renormalize solution as defined in \cite{JF2015}. We have the following theorem.
\begin{thm}\label{TS: renormalize solution to}
Let $f^N:=(f_0^N,\cdots, f_N^N)$ be the $L^1((0,T); \left(W^{1,1}(\Omega)\right)^{N+1})$ solution to \eqref{truncated system}, obtained as the limit of $f_{\epsilon}^N:=(f_{0,\epsilon}^N,\cdots, f_{N,\epsilon}^N)$, which satisfies \eqref{regularized truncated system}. Let $\psi:\mathbb{R}^{N+1}\to\mathbb{R}$ be a smooth function with compactly supported first derivatives. Then the following holds
\begin{align}
    -&\int_{\Omega} \psi(f^N(0))\xi(0)\, dx-\int_0^T\int_{\Omega} \psi(f^N) \partial_t \xi \, dx\, dt \label{TS: renormalized solution, intermediate 1}\\
    =& - \sum_{i,k=0}^N \int_0^T\int_{\Omega} d_i\xi \partial_i\partial_k \psi(f^N) \nabla f_i^N\cdot\nabla f_k^N \, dx\, dt- \sum_{i=0}^N \int_0^T\int_{\Omega} d_i \partial_i \psi(f^N) \nabla f_i^N\cdot\nabla \xi\, dx\, dt \nonumber\\
    &+\sum_{i=0}^N \int_0^T\int_{\Omega} \partial_i \psi(f^N) Q_i^N(f^N)\xi\, dx\, dt, \nonumber
\end{align}
where, for $0\leq i\leq N$, $Q_i^N$ be the source term to the truncated system \eqref{truncated system} as defined in \eqref{truncated system}.    
\end{thm}
\begin{proof}
    We revisit \eqref{TS: composition with the truncation of identity} in Proposition \ref{TS: composition with the truncated smooth function}. We have that 
    \begin{align}
 -\int_{\Omega} &\phi^{\Lambda}_i \left(f_{0}^N\right) \xi(0)\, dx -\int_{0}^T\int_{\Omega} \phi^{\Lambda}_i \left(f^N \right) \partial_t\xi \, dx\, dt   \label{revisit 1}\\
=& \sum_{j=0}^N \int_0^T \int_{\Omega} \partial_j \phi_i^{\Lambda}\left(f^N\right) Q^N_{j}  \xi \, dx\, dt- \sum_{j=0}^N \int_0^T \int_{\Omega} d_j \partial_j \phi_i^{\Lambda}\left(f^N\right) \nabla f_{j}^N \cdot \nabla \xi \, dx\, dt \nonumber\\
    & \ \ \ \ \ + \int_0^T \int_{\Omega} \xi \, d\mu_i^{\Lambda}(x,t),\nonumber    
\end{align}
We will show that $Q_i^N(f^N)\in L^1((0,T)\times\Omega)$ uniformly in `$N$'. The following computation holds for $i=1,\cdots N$:
\begin{align}
    \int_{0}^T\int_{\Omega} | Q_i^N(f^N)|\, dx\, dt \leq & \int_0^T \int_{\Omega} \left(f_{i+1}^N \sum_{j=0}^{N-1} K_{i+1,j} f_j^N+f_{i}^N \sum_{j=0}^{N-1} K_{i,j}f_j^N\right) \, dx\, dt \label{uniform bound source in N}\\
    +& \int_0^T\int_{\Omega} \left( f_i^N\sum_{j=1}^N K_{j,i}f_j^N+f_{i-1}^N \sum_{j=1}^N K_{j,i-1}f_j^N\right)\, dx\, dt, \nonumber\\
    \leq \sup_{i,j}\left\{\frac{K_{i,j}}{(i+1)(j+1)}\right\} \int_0^T \int_{\Omega} &\left(\sum_{j=0}^{N-1} (i+2)(j+1)f_{i+1}^N f_j^N+ \sum_{j=0}^{N-1} (i+1)(j+1)f_i^Nf_j^N\right) \, dx\, dt \nonumber \\
    +\sup_{i,j}\left\{\frac{K_{i,j}}{(i+1)(j+1)}\right\} \int_0^T\int_{\Omega} &\left( \sum_{j=1}^N (i+1)(j+1)f_i^Nf_j^N+ \sum_{j=1}^N i(j+1) f_{i-1}^N f_j^N\right)\, dx\, dt \nonumber \\
    \lesssim & \sup_{i,j}\left\{\frac{K_{i,j}}{(i+1)(j+1)}\right\} \left( \frac{d^*}{d_*} \right)
\left\|
\sum_{i=0}^{\infty} (i+1)f_{i,0}
\right\|_{L^2(\Omega)}^{2}, \nonumber
\end{align}
where we use the estimate \eqref{L2 estimate: truncated system} from Theorem \ref{existence of solution: truncated system}, in the last line. Furthermore, thanks to  Lemma \ref{lem:trunc-energy}, we can show, for $0\leq i\leq N$, $v_i:=\phi_i^{\Lambda}(f^N)\in L^1((0,T); L^1(\Omega))\cap L^2((0,T);H^1(\Omega))$, where $\phi_i^{\Lambda}$ is as defined in \ref{truncation to identity multi D}. For $0\leq i\leq N$, let us consider
\begin{equation*}
    \left \{
    \begin{aligned}
        &v_0:=(v_{0,0},\cdots,\cdots, v_{0,N}),\\
        & v_{0,i}:=\phi^{\Lambda}_i(f_{i}^N(0))\in L^1(\Omega),\\
        & v_i:= \phi_i^{\Lambda}(f^N)\in L^1((0,T); L^1(\Omega))\cap L^2((0,T);H^1(\Omega)),\\
        & z_i:= -\sum_{j=0}^N d_j \partial_j \phi^{\Lambda}_i(f^N) \nabla f_j^N \in L^2((0,T); L^2(\Omega)),\\
        &w_i:=\sum_{j=0}^N \partial_j \phi_i^{\Lambda}(f^N) Q_j^N(f^N) \in L^1((0,T);L^1(\Omega)),\\
        & \nu_i \in RM([0,T)\times\Omega).
    \end{aligned}
    \right .
\end{equation*}
Hence applying Lemma \ref{TS: auxiliary lemma, renormalized solution} in \eqref{revisit 1}, we obtain that
\begin{align}
         \Bigg| &-\int_0^T\int_{\Omega} \psi(\phi^{\Lambda}(f^N)) \partial_t \xi \, dx\, dt- \int_{\Omega} \psi(\phi^{\Lambda}(f^N(0)))\xi(0)\, dx \nonumber \\
         &- \sum_{i=0}^N\sum_{j=0}^N \int_0^T\int_{\Omega} \xi \partial_i \psi(\phi^{\Lambda}(f^N)) \partial_j\phi_i^{\Lambda}(f^N) Q_j^N\, dx\, dt \nonumber\\
        &+\sum_{i=0}^N \sum_{j=0}^N\int_0^T\int_{\Omega} d_j\partial_i \psi(\phi^{\Lambda}(f^N)) \partial_j \phi_i^{\Lambda}(f^N)\nabla f_j^N\cdot\nabla\xi \, dx\, dt \nonumber
         \end{align}
             \begin{align}\label{TS: before ending truncation}
        &+\sum_{i=0}^N\sum_{k=0}^N \sum_{j=0}^N \sum_{l=0}^N\int_0^T\int_{\Omega} \xi d_j\partial_{i}\partial_{k}\psi(\phi^{\Lambda}(f^N)) \partial_j \phi_i^{\Lambda}(f^N)\nabla f_j^N\cdot\partial_l \phi_k^{\Lambda}(f^N) \nabla f^N_l\, dx\, dt\Bigg| \nonumber\\
        &\leq C(\Omega) \| \xi\|_{L^{\infty}([0,T)\times\Omega)} \left( \sup_{x\in \mathbb{R}^{N+1}}|D\psi(x)|\right) \sum_{i=0}^N \| \nu_i \|_{RM([0,T)\times\Omega)}, 
\end{align}
where $\phi^{\Lambda}(f^N):=(\phi_0^{\Lambda}(f^N),\cdots, \phi_N^{\Lambda}(f^N))$. Now, we pass the limit $\Lambda\to+\infty$. Note that thanks to Proposition \ref{TS: composition with the truncated smooth function}, r.h.s. of the above term goes to zero as $\Lambda\to\infty$. Using the fact that $\displaystyle{\lim_{\Lambda\to\infty}\phi_i^{\Lambda}(v)=v_i}$, $\displaystyle{\lim_{\Lambda\to\infty}\partial_j\phi_i^{\Lambda}(v)=\delta_{i,j}}$ and that the first derivatives of $\phi_i^{\Lambda}$ is uniformly bounded (see \ref{truncation to identity multi D}), we can pass the limit $\Lambda\to+\infty$ in the first term  \eqref{TS: before ending truncation} using dominated convergence theorem. It remains to deal with the fourth term of \eqref{TS: before ending truncation}. Thanks to the properties of $\phi^{\Lambda}_i$ (see \ref{truncation to identity multi D}), the pointwise limit holds
\begin{align}
\lim_{\Lambda\to\infty} \sum_{i=0}^N\sum_{k=0}^N \sum_{j=0}^N \sum_{l=0}^N \xi d_j\partial_{i}\partial_{k}&\psi(\phi^{\Lambda}(f^N)) \partial_j \phi_i^{\Lambda}(f^N)\nabla f_j^N\cdot\partial_l \phi_k^{\Lambda}(f^N) \nabla f^N_l  \label{TS:pointwise limit}\\
& = \sum_{i,k=0}^N \xi d_i \partial_i\partial_k \psi(f^N)\nabla f_i^N \nabla f_k^N, \nonumber
\end{align}
where we use the fact that $\displaystyle{\lim_{\Lambda\to\infty}\partial_j\phi_i^{\Lambda}(v)=\delta_{i,j}}$ (see \ref{truncation to identity multi D}). Furthermore, Let the support of $D(\psi)$ lies inside $B(0,M)$. Hence the function inside the fourth term in  \eqref{TS: before ending truncation} is non zero only when $\sum_{i=0}^N \phi_i^{\Lambda}(f^N)\leq M$. For large $\Lambda>M$, We intend to show that the function inside the fourth term in  \eqref{TS: before ending truncation} is non zero only when $\sum_{i=0}^M{f_i}\leq M$. Assume $\sum_{i=0}^M{f_i}\geq M$, then from \eqref{truncation to identity multi D}, we have that  $\sum_{i=0}^N \phi_i^{\Lambda}(f^N)\geq \min\left\{\Lambda,\sum_{i=0}^M{f_i}\right\} \geq M$. Here we use the fact that, for all $K_3>0$, $\sum_{i=0}^N v_j>K_3$, implies $\sum_{i=0}^N \phi_i^{\Lambda}(v)\geq \min\{K_3,\Lambda\}$, for all $\Lambda>0$ and $v\in(\mathbb{R}_{\geq 0})^{N+1}$. It yields that the function inside the fourth term in \eqref{TS: before ending truncation} vanish when $\sum_{i=0}^N f_i^N>M$. Hence the following computation holds
\begin{align*}
  &\sum_{i=0}^N\sum_{k=0}^N \sum_{j=0}^N \sum_{l=0}^N \xi d_j\partial_{i}\partial_{k}\psi(\phi^{\Lambda}(f^N)) \partial_j \phi_i^{\Lambda}(f^N)\nabla f_j^N\cdot\partial_l \phi_k^{\Lambda}(f^N) \nabla f^N_l   \\
  =4 \sum_{i=0}^N\sum_{k=0}^N& \sum_{j=0}^N \sum_{l=0}^N \Bigg(\xi d_j\sqrt{\left( f_j^N\right)} \sqrt{\left( f_l^N\right)}\partial_{i}\partial_{k}\psi(\phi^{\Lambda}(f^N)) \partial_j \phi_i^{\Lambda}(f^N)\nabla \sqrt{\left(f_j^N\right)}\cdot \\&\hspace{8cm}\partial_l \phi_k^{\Lambda}(f^N) \nabla \sqrt{\left(f^N_l\right)}\Bigg).
\end{align*}
Assume the support of $D\psi$ contained in $B(0,M)$, for some $M>0$, the above relation yields
\begin{align*}
  \Bigg|\sum_{i=0}^N\sum_{k=0}^N &\sum_{j=0}^N \sum_{l=0}^N  \xi d_j\partial_{i}\partial_{k}\psi(\phi^{\Lambda}(f^N)) \partial_j \phi_i^{\Lambda}(f^N)\nabla f_j^N\cdot\partial_l \phi_k^{\Lambda}(f^N) \nabla f^N_l  \Bigg|
   \end{align*}
             \begin{align*}
  \leq &8MNK_2^2\sup_{i\geq 0}\{d_j\}\|\xi\|_{L^{\infty}([0,T)\times\Omega)} \sup_{v\in\mathbb{R}^{N+1}}\left\{D^2(\psi)(v) \right\}\sum_{j=0}^N \sum_{l=0}^N \nabla\sqrt{\left(f_j^N\right)}\cdot \nabla \sqrt{\left(f^N_l\right)},
\end{align*}
where thanks to Theorem \ref{existence of solution: truncated system}, the last quantity is integrable and $K_2$ is as defined in \ref{truncation to identity multi D}. Hence, thanks to dominated convergence theorem and the relation \eqref{TS:pointwise limit}, we obtain
\begin{align*}
    \lim_{\Lambda\to\infty} \sum_{i=0}^N\sum_{k=0}^N \sum_{j=0}^N \sum_{l=0}^N \int_0^T \int_{\Omega}\xi d_j\partial_{i}\partial_{k}&\psi(\phi^{\Lambda}(f^N)) \partial_j \phi_i^{\Lambda}(f^N)\nabla f_j^N\cdot\partial_l \phi_k^{\Lambda}(f^N) \nabla f^N_l \, dx\, dt \label{TS:pointwise limit}\\
& = \sum_{i,k=0}^N \int_0^T\int_{\Omega} \xi d_i \partial_i\partial_k \psi(f^N)\nabla f_i^N \nabla f_k^N \, dx\, dt.
\end{align*}
\end{proof}
Before proceeding further, we define a special one-dimensional function that serves as a truncation of the identity function.
\subsection{Truncation to identity}\label{truncation to identity} Let $\Lambda>0$. We consider the function $\Phi^{\Lambda}: \mathbb{R}_{\geq 0}\to\mathbb{R}_{\geq 0}$, satisfies
\begin{itemize}
    \item [$\bullet$] $\Phi^{\Lambda}\in C^2(\mathbb{R}_{\geq 0})$.
    \item[$\bullet$] There exists $K_1>0$, such that $\displaystyle{\sqrt{v}\sqrt{v}\left| \frac{d^2}{dv^2}(\Phi^{\Lambda}(v))\right|\leq K_1}$, \ for all $v\in\mathbb{R}_{\geq 0}$ and for all $\Lambda$.
    \item[$\bullet$] For each $\Lambda>0$, $\mathrm{supp}\left\{\frac{d}{dv}(\Phi^{\Lambda})\right\}$ is compact, say  $\mathrm{supp}\left\{\frac{d}{dv}(\Phi^{\Lambda})\right\} \subset B(0,\Lambda^*)$. 
    \item[$\bullet$] For all $v\geq 0$, $\displaystyle{\lim_{\Lambda\to\infty} \frac{d}{dv} \Phi^{\Lambda}(v)=1}$.
    \item[$\bullet$] There exists $K_2\geq 0$, such that $\displaystyle{\left| \frac{d}{dv} \Phi^{\Lambda}(v)\right|\leq K_2, \ \forall\, v\in\mathbb{R}_{\geq 0},  \ \forall\, \Lambda>0}$.
    \item[$\bullet$] $\Phi^{\Lambda}(v)=v$, \ $\forall\, v\in\mathbb{R}_{\geq 0)}$, where $v\leq \Lambda$.
    \item[$\bullet$] $\displaystyle{\lim_{\Lambda\to\infty} \sup_{|v|\leq K}\left| \frac{d^2}{dv^2} \Phi^{\Lambda}(v)\right|=0, \ \forall \, K>0}$. 
\end{itemize}
 Such a function can be constructed following Remark \ref{existence of such truncation to identity}. We have the following proposition.
\begin{prop}\label{TS: composition with a special truncation to identity}
    Let $\eta>0$ be arbitrary. For $0\leq i\leq N$, let $f_i^N$ be the solution to \eqref{truncated system}, obtained as the $L^1((0,T);W^{1,1}(\Omega))$ limit of the solutions of \eqref{regularized truncated system}. Let $\Phi^{\Lambda}$ be as defined in \ref{truncation to identity}. Let us fix an index $i=I$. Let us denote $f_{-1}^N=Q_{-1}^N=d_{-1}=0$. Then, for all $\xi\in C_c^{\infty}([0,T)\times\Omega)$, the following holds
\begin{align*}
    &-\int_{\Omega} \Phi^{\Lambda} \left(f_I^N+\eta\left(f_{I+1}^N+f_{I-1}^N\right)(0)\right) \xi(0) \, dx -\int_0^T\int_{\Omega} \Phi^{\Lambda} \left(f_I^N+\eta\left(f_{I+1}^N+f_{I-1}^N\right)\right) \partial_t\xi \, dx\, dt\\
    &= \int_0^T \int_{\Omega} (\Phi^{\Lambda})' \left(f_I^N+\eta\left(f_{I+1}^N+f_{I-1}^N\right)\right) \left[Q_I^N+\eta\left(Q_{I+1}^N+Q_{I-1}^N\right)\right] \xi \, dx\, dt\\
    &- \int_0^T \int_{\Omega} d_I (\Phi^{\Lambda})' \left(f_I^N+\eta\left(f_{I+1}^N+f_{I-1}^N\right)\right) \left(\nabla f_I^N+\eta\left(\nabla f_{I+1}^N+\nabla f_{I-1}^N\right)\right)\cdot \nabla \xi \, dx\, dt\\
    & -\int_0^T\int_{\Omega} d_{I}(\Phi^{\Lambda})'' \left(f_I^N+\eta\left(f_{I+1}^N+f_{I-1}^N\right)\right) \left| \nabla f_I^N+\eta\left(\nabla f_{I+1}^N+\nabla f_{I-1}^N\right)\right|^2 \xi \, dx\, dt\\
    &- \int_0^T \int_{\Omega} \eta(d_{I-1}-d_I)\Bigg((\Phi^{\Lambda})''\left(f_I^N+\eta\left(f_{I+1}^N+f_{I-1}^N\right)\right) \nabla f_{I-1}^N \cdot
     \end{align*}
             \begin{align*}
    &\hspace{7cm}\left(\nabla f_I^N+\eta\left(\nabla f_{I+1}^N+\nabla f_{I-1}^N\right)\right) \xi \Bigg)\, dx\, dt \\
    &- \int_0^T \int_{\Omega} \eta(d_{I+1}-d_I)\Bigg((\Phi^{\Lambda})'' \left(f_I^N+\eta\left(f_{I+1}^N+f_{I-1}^N\right)\right) \nabla f_{I+1}^N \cdot \\
    &\hspace{7cm}\left(\nabla f_I^N+\eta\left(\nabla f_{I+1}^N+\nabla f_{I-1}^N\right)\right) \xi\Bigg) \, dx\, dt \\
    &- \int_0^T \int_{\Omega} \eta (d_{I-1}-d_I)  (\Phi^{\Lambda})' \left(f_I^N+\eta\left(f_{I+1}^N+f_{I-1}^N\right)\right) \nabla f_{I-1}^N \cdot\nabla \xi \, dx\, dt\\
     &- \int_0^T \int_{\Omega} \eta (d_{I+1}-d_I)  (\Phi^{\Lambda})' \left(f_I^N+\eta\left(f_{I+1}^N+f_{I-1}^N\right)\right) \nabla f_{I+1}^N \cdot \nabla \xi \, dx\, dt.
\end{align*}
\end{prop}
\begin{proof}
 Let us choose $\psi(v)=\Phi^{\Lambda}(v_I+\eta(v_{I+1}+v_{I-1}))$ in Theorem \ref{TS: renormalize solution to}, where $v:=(v_0,\cdots,v_N)\in\mathbb{R}^{N+1}$ with $v_{-1}=0$. From  \eqref{TS: renormalized solution, intermediate 1}, we obtain
    \begin{align*}
         &-\int_{\Omega} \Phi^{\Lambda} \left(f_I^N+\eta\left(f_{I+1}^N+f_{I-1}^N\right)(0)\right) \xi(0) \, dx -\int_0^T\int_{\Omega} \Phi^{\Lambda} \left(f_I^N+\eta\left(f_{I+1}^N+f_{I-1}^N\right)\right) \partial_t\xi \, dx\, dt\\
         &= \int_0^T\int_{\Omega} (\Phi^{\Lambda})' \left(f_I^N+\eta\left(f_{I+1}^N+f_{I-1}^N\right)\right) Q_I^N(f^N)\xi\, dx\, dt\\
         &+\int_0^T\int_{\Omega} \eta (\Phi^{\Lambda})' \left(f_I^N+\eta\left(f_{I+1}^N+f_{I-1}^N\right)\right) Q_{I+1}^N(f^N)\xi\, dx\, dt\\
        & +\int_0^T\int_{\Omega} \eta(\Phi^{\Lambda})' \left(f_I^N+\eta\left(f_{I+1}^N+f_{I-1}^N\right)\right) Q_{I-1}^N(f^N)\xi\, dx\, dt\\
        &- \int_0^T \int_{\Omega} d_I (\Phi^{\Lambda})' \left(f_I^N+\eta\left(f_{I+1}^N+f_{I-1}^N\right)\right) \nabla f_I^N\cdot\nabla \xi\, dx\, dt\\
        &- \int_0^T \int_{\Omega} \eta d_{I+1} (\Phi^{\Lambda})' \left(f_I^N+\eta\left(f_{I+1}^N+f_{I-1}^N\right)\right) \nabla f_{I+1}^N\cdot\nabla \xi\, dx\, dt\\
        &- \int_0^T \int_{\Omega} \eta d_{I-1} (\Phi^{\Lambda})' \left(f_I^N+\eta\left(f_{I+1}^N+f_{I-1}^N\right)\right) \nabla f_{I-1}^N\cdot\nabla \xi\, dx\, dt\\
        & - \int_0^T \int_{\Omega} d_I \xi (\Phi^{\Lambda})''\left(f_I^N+\eta\left(f_{I+1}^N+f_{I-1}^N\right)\right) \left| \nabla f_I^N\right|^2\, dx\, dt\\
         & - \int_0^T \int_{\Omega} \eta d_I \xi (\Phi^{\Lambda})'' \left(f_I^N+\eta\left(f_{I+1}^N+f_{I-1}^N\right)\right) \nabla f_I^N \cdot \left( \nabla f_{I+1}^N+\nabla f_{I-1}^N\right) \, dx\, dt
             \end{align*}
             \begin{align*}
         & - \int_0^T \int_{\Omega}  \eta d_{I+1} \xi (\Phi^{\Lambda})''\left(f_I^N+\eta\left(f_{I+1}^N+f_{I-1}^N\right)\right) \nabla f_I^N \cdot \nabla f_{I+1}^N \, dx\, dt\\
         & - \int_0^T \int_{\Omega} \eta d_{I-1} \xi (\Phi^{\Lambda})''\left(f_I^N+\eta\left(f_{I+1}^N+f_{I-1}^N\right)\right) \nabla f_I^N \cdot \nabla f_{I-1}^N\, dx\, dt\\
         & - \int_0^T \int_{\Omega} \eta^2  d_{I+1} \xi (\Phi^{\Lambda})'' \left(f_I^N+\eta\left(f_{I+1}^N+f_{I-1}^N\right)\right) \left(\left| \nabla f_{I+1}\right|^2+\nabla f_{I-1}^N\cdot \nabla f_{I+1}^N\right) \, dx\, dt\\
        & - \int_0^T \int_{\Omega} \eta^2  d_{I-1} \xi (\Phi^{\Lambda})''\left(f_I^N+\eta\left(f_{I+1}^N+f_{I-1}^N\right)\right) \left(\left| \nabla f_{I-1}\right|^2+\nabla f_{I-1}^N\cdot \nabla f_{I+1}^N\right) \, dx\, dt.
    \end{align*}
Rearrangement of the above terms yields the result.     
\end{proof}
\section{Existence to the diffusive nonlinear exchange driven growth model}
In this section we devote ourselves to find compactness of the sequence $\{f_i^N\}$ as $N\to \infty$. We recall a compactness result from\cite{baras1984, bothe2010, Pierre2010}, which we
use to establish the convergence of the sequence of solutions to the truncated system \eqref{truncated system}.
\begin{lem}\label{lem:compactness}
Let $d>0$. The mapping $( \mathcal{F}^{\rm{in}}, \mathcal{B})\mapsto \mathcal{F}$, where $\mathcal{F}$ is the solution of
\begin{equation}\label{eq:compactness_heat}
\begin{cases}
\partial_t \mathcal{F} - d \Delta \mathcal{F} = \mathcal{B} & \text{in } (0,T)\times \Omega \\
\nabla \mathcal{F}\cdot n = 0 & \text{on }  (0,T) \times \partial\Omega \\
\mathcal{F}(0,\cdot) =  \mathcal{F}^{\rm{in}} & \text{in } \Omega,
\end{cases}
\end{equation}
is compact from $L^{1}(\Omega)\times L^{1}((0,T)\times \Omega)$ into
$L^{1}((0,T)\times \Omega)$, and even into
$L^{1}\big((0,T); W^{1,1}(\Omega)\big)$.
\end{lem}
We are now well equipped to establish compactness for the solutions to the truncated system \eqref{truncated system}. More precisely, we have the following lemma.
\begin{lem}\label{compactness to the solution of TS}
For $0\leq i\leq N$, let $f_i^N$ be the solutions to the truncated system \eqref{truncated system}, obtained as $L^1((0,T);W^{1,1}(\Omega))$ limit of the solutions of \eqref{regularized truncated system}. Then, for all $i\in\mathbb{N}$, there exists $f_i$ such that $f_i^N\to f_i$ in  $L^1((0,T);W^{1,1}(\Omega))$ as $N\to \infty$.  Furthermore, the following holds
    \begin{equation}
\label{L2 estimate: main system}
\left \{
\begin{aligned}
& \int_0^T \int_{\Omega} \left( \sum_{i=0}^{\infty} (i+1) f_i(t,x)\right)\, dx \leq \left\| \sum_{i=0}^N (i+1) f_{i,0}\right\|_{L^1(\Omega)},\ \ \forall\, t\geq 0,\\  
&\int_0^T \int_\Omega
\left( \sum_{i=0}^{\infty} i  f_{i}(t,x) \right)
\left( \sum_{i=0}^{\infty} i f_{i}(t,x) \right)
\,dx\,dt
\;\lesssim\;
 \frac{d^*}{d_*} 
\left\|
\sum_{i=0}^{\infty} i f_{i,0}
\right\|_{L^2(\Omega)}^{2},\\
&\int_0^T \int_\Omega
\left( \sum_{i=0}^{\infty}   f_{i}(t,x) \right)
\left( \sum_{i=0}^{\infty}  f_{i}(t,x) \right)
\,dx\,dt
\;\lesssim\;
 \frac{d^*}{d_*} 
\left\|
\sum_{i=0}^{\infty} f_{i,0}
\right\|_{L^2(\Omega)}^{2},\\
&\int_0^T \int_\Omega
\left( \sum_{i=0}^{\infty}  (i+1) f_{i}(t,x) \right)
\left( \sum_{i=0}^{\infty}  (i+1)f_{i}(t,x) \right)
\,dx\,dt
\\
&\hspace{6cm}\;\lesssim\;
 \frac{d^*}{d_*}
\left\|
\sum_{i=0}^{\infty} (i+1)f_{i,0}
\right\|_{L^2(\Omega)}^{2},\\
&\hspace{5cm}\sum_{i=0}^{\infty} \int_0^T\int_{\Omega} \frac{|\nabla f_{i}|^2}{f_{i}} \leq M_{\mathcal{F}},
\end{aligned}
\right .
\end{equation}
where $M_{\mathcal{F}}$ is a positive constant.
\end{lem}
\begin{proof}
    Let us fix an index $i\in\mathbb{N}$. the function $f_i^N$ satisfies the following equation:
\begin{equation*}
\begin{cases}
\partial_t f_i^N - d_i \Delta f_i^N = Q_i^N & \text{in } (0,T)\times \Omega \\
\nabla f_i^N\cdot n = 0 & \text{on }  (0,T) \times \partial\Omega \\
f_i^N(0,\cdot) = f_{i,0} & \text{in } \Omega.
\end{cases}
\end{equation*}
Thanks to \eqref{uniform bound source in N}, we have that 
\[
\int_0^T \int_{\Omega} \left| Q_i^N(f^N)\right| \, dx\, dt \lesssim \sup_{i,j}\left\{\frac{K_{i,j}}{(i+1)(j+1)}\right\} \left( \frac{d^*}{d_*} \right)
\left\|
\sum_{i=0}^{\infty} (i+1)f_{i,0}
\right\|_{L^2(\Omega)}^{2}.
\]
Hence, Lemma \ref{lem:compactness} yields the result. Result of the estimate follows from Fatou's lemma applied on \eqref{L2 estimate: truncated system}.
\end{proof}
We now try to pass $N \to \infty$ in Proposition \ref{TS: composition with a special truncation to identity}. We have the following lemma.
\begin{lem}\label{N to infinity inside truncation}
    For $i\in \mathbb{N}\cup\{0\}$, let $f_i$ be the function  described in Lemma \ref{compactness to the solution of TS}. Let $\eta\in(0,1)$ be arbitrary and $\Phi^{\Lambda}$ be as defined in \ref{truncation to identity}. Let us fix an index $i=I$. Let us denote $f_{-1}=Q_{-1}=d_{-1}=0$. Then for all $\xi\in C_c^{\infty}([0,T)\times\Omega)$, the following holds
    \begin{align*}
       \Bigg| & -\int_{\Omega} \Phi^{\Lambda} \left(f_{I,0}+\eta\left(f_{I+1,0}+f_{I-1,0}\right)\right) \xi(0)\, dx-\int_0^T \int_{\Omega} \Phi^{\Lambda} \left( f_I +\eta\left( f_{I+1}+f_{I-1}\right)\right) \partial_t \xi\, dx\, dt\\
       & -\int_0^T \int_{\Omega} \left(\Phi^{\Lambda}\right)'\left( f_I+\eta\left(f_{I+1}+f_{I-1}\right)\right) Q_I \xi \, dx\, dt\\
       &+ \int_0^T \int_{\Omega} \left(\Phi^{\Lambda}\right)'\left( f_I+\eta\left(f_{I+1}+f_{I-1}\right)\right) \nabla f_I \cdot \nabla\xi \, dx\, dt + \int_0^T\int_{\Omega} \nu_{I,\eta}^{\Lambda} \xi \, dx\, dt\Bigg| \\
       & \qquad \leq \texttt{C} \eta^{\frac{1}{2}} \|\xi\|_{W^{1,\infty}((0,T)\times\Omega)},
    \end{align*}
    for some constant $\texttt{C}>0$ (may depend on $\Lambda$) and  radon measure $\nu_{I,\eta}^{\Lambda}\in RM([0,T)\times\Omega)$. Here $Q_I(f)$ is as defined in \eqref{main equation} and 
    \[
    \lim_{\Lambda\to \infty} \left|\nu_{I,\eta}^{\Lambda} \right|([0,T)\times\Omega)\to 0, \quad \forall I\in\mathbb{N}, \, \eta>0.
    \]
\end{lem}
\begin{proof}
    We consider the terms in Proposition \ref{TS: composition with a special truncation to identity}. The first term is always constant and it is exactly as $\displaystyle{-\int_{\Omega} \Phi^{\Lambda} \left(f_{I,0}+\eta\left(f_{I+1,0}+f_{I-1,0}\right)\right) \xi(0)\, dx}$. Following the calculation 
    \begin{align*}
        \int_0^T \int_{\Omega} & \left| \Phi^{\Lambda}\left( f_I^N+\eta\left( f_{I+1}^N+f_{I-1}^N\right)\right) \partial_t \xi - \Phi^{\Lambda}\left( f_I+\eta\left( f_{I+1}+f_{I-1}\right)\right) \partial_t \xi\right| \, dx\, dt\\
        & \leq K_2\| \partial_t \xi\|_{L^{\infty}((0,T)\times\Omega)} \int_0^T \int_{\Omega} \left | f_I^N+\eta\left( f_{I+1}^N+f_{I-1}^N\right) -f^N- \eta \left(f_{I+1}+f_{I-1}\right)\right| \to 0,
    \end{align*}
    we conclude that the second term in Proposition \ref{TS: composition with a special truncation to identity} converges to 
    \[    \displaystyle{\int_0^T\int_{\Omega}\Phi^{\Lambda}\left( f_I+\eta\left( f_{I+1}+f_{I-1}\right)\right) \partial_t \xi \, dx\, dt.}
    \]
    Next we consider the term $\displaystyle{\int_0^T \int_{\Omega} \left(\Phi^{\Lambda}\right)'\left( f_I^N+\eta\left(f_{I+1}^N+f_{I-1}^N\right)\right) Q_I^N (f^N_I) \xi \, dx\, dt.}$ We present our calculation for $I\geq 1$, for $I=0$ similar calculation follows. We like to note that, for $I\geq 1$
\begin{align}
Q_I^N:=f_{I+1}^N \sum_{j=0}^{N-1} K_{I+1,j} f_j^N
-& f_I^N \sum_{j=0}^{N-1} K_{I,j} f_j^N - f_I^N \displaystyle\sum_{j=1}^{N} K_{j,I} f_j^N
+ f_{I-1}^N \displaystyle\sum_{j=1}^{N} K_{j,I-1} f_j^N \nonumber\\
&= \mathcal{L}_1^N+\mathcal{L}_2^N+\mathcal{L}_3^N+\mathcal{L}_4^N. \label{truncated to original, limit pass, intermediate 1}
\end{align}
The following calculation holds
\begin{align}
    \int_0^T \int_{\Omega} \Bigg| &\left(\Phi^{\Lambda}\right)'\left( f_I^N+\eta\left(f_{I+1}^N+f_{I-1}^N\right)\right)f_{I+1}^N \sum_{j=0}^{N-1} K_{I+1,j} f_j^N \nonumber\\
    &-\left(\Phi^{\Lambda}\right)'\left( f_I+\eta\left(f_{I+1}+f_{I-1}\right)\right) f_{I+1} \sum_{j=0}^{\infty} K_{I+1,j} f_j\Bigg|\, dx\, dt \nonumber\\
    \leq \int_0^T \int_{\Omega} \Bigg| &\left(\Phi^{\Lambda}\right)'\left( f_I^N+\eta\left(f_{I+1}^N+f_{I-1}^N\right)\right)f_{I+1}^N \sum_{j=0}^{N-1} K_{I+1,j} f_j^N \nonumber \\
    &-\left(\Phi^{\Lambda}\right)'\left( f_I^N+\eta\left(f_{I+1}^N+f_{I-1}^N\right)\right) f_{I+1}^N \sum_{j=0}^{\infty} K_{I+1,j} f_j\Bigg|\, dx\, dt \nonumber \\
    + \int_0^T \int_{\Omega} \Bigg| &\left(\Phi^{\Lambda}\right)'\left( f_I^N+\eta\left(f_{I+1}^N+f_{I-1}^N\right)\right)f_{I+1}^N \sum_{j=0}^{\infty} K_{I+1,j} f_j \nonumber \\
    &-\left(\Phi^{\Lambda}\right)'\left( f_I+\eta\left(f_{I+1}+f_{I-1}\right)\right) f_{I+1} \sum_{j=0}^{\infty} K_{I+1,j} f_j\Bigg|\, dx\, dt. \label{truncated to original, limit pass, intermediate 2}
\end{align}
Thanks to \ref{truncation to identity}, each term in the second integral is dominated by the function \newline
$\displaystyle{K_2 \frac{\Lambda^*}{\eta}\sum_{j=0}^{\infty} K_{I+1,j} f_j}$, which is integrable thanks to moment conservation property
\[
\int_0^T \int_{\Omega}\sum_{j=0}^{\infty} K_{I+1,j} f_j \, dx\, dt \leq \sup_{j\geq 0} \left\{ \frac{K_{I+1,j}}{j+1}\right\} \int_0^T \int_{\Omega}\sum_{j=0}^{\infty} (j+1) f_j\, dx\, dt<+\infty.
\]
Furthermore thanks to the pointwise limit
\begin{align*}
\lim_{N\to\infty}&\left(\Phi^{\Lambda}\right)'\left( f_I^N+\eta\left(f_{I+1}^N+f_{I-1}^N\right)\right)f_{I+1}^N \sum_{j=0}^{\infty} K_{I+1,j} f_j \\
&=\left(\Phi^{\Lambda}\right)'\left( f_I+\eta\left(f_{I+1}+f_{I-1}\right)\right) f_{I+1}^N \sum_{j=0}^{\infty} K_{I+1,j} f_j,
\end{align*}
we obtain that the last term of \eqref{truncated to original, limit pass, intermediate 2} converges to zero. To handle  the first term of \eqref{truncated to original, limit pass, intermediate 2}, we fix an integer $N^*>0$. The following calculation holds
\begin{align}
     \int_0^T \int_{\Omega} \Bigg| &\left(\Phi^{\Lambda}\right)'\left( f_I^N+\eta\left(f_{I+1}^N+f_{I-1}^N\right)\right)f_{I+1}^N \sum_{j=0}^{N-1} K_{I+1,j} f_j^N \nonumber \\
    &-\left(\Phi^{\Lambda}\right)'\left( f_I^N+\eta\left(f_{I+1}^N+f_{I-1}^N\right)\right) f_{I+1}^N \sum_{j=0}^{\infty} K_{I+1,j} f_j\Bigg|\, dx\, dt \nonumber \\
    \leq \int_0^T \int_{\Omega} \Bigg| &\left(\Phi^{\Lambda}\right)'\left( f_I^N+\eta\left(f_{I+1}^N+f_{I-1}^N\right)\right)f_{I+1}^N \sum_{j=0}^{N^*} K_{I+1,j} f_j^N \nonumber \\
    &-\left(\Phi^{\Lambda}\right)'\left( f_I^N+\eta\left(f_{I+1}^N+f_{I-1}^N\right)\right) f_{I+1}^N \sum_{j=0}^{N^*} K_{I+1,j} f_j\Bigg|\, dx\, dt \nonumber\\
    + \sup_{j\geq N^*}\left\{ \frac{K_{I+1,j}}{j+1}\right\}\int_0^T \int_{\Omega} \Bigg| &\left(\Phi^{\Lambda}\right)'\left( f_I^N+\eta\left(f_{I+1}^N+f_{I-1}^N\right)\right)f_{I+1}^N\Bigg| \sum_{j=N^*}^{\infty} (j+1) f_j^N \, dx\, dt\nonumber\\
    + \sup_{j\geq N^*}\left\{ \frac{K_{I+1,j}}{j+1}\right\}\int_0^T \int_{\Omega} \Bigg| &\left(\Phi^{\Lambda}\right)'\left( f_I^N+\eta\left(f_{I+1}^N+f_{I-1}^N\right)\right)f_{I+1}^N\Bigg| \sum_{j=N^*}^{\infty} (j+1) f_j \, dx\, dt.\nonumber
\end{align}
Letting $N\to \infty$, the first term of the right hand side goes to zero. Hence thanks to momentum conservation property, we have that 
\begin{align*}
    \int_0^T \int_{\Omega} \Bigg| &\left(\Phi^{\Lambda}\right)'\left( f_I^N+\eta\left(f_{I+1}^N+f_{I-1}^N\right)\right)f_{I+1}^N \sum_{j=0}^{N-1} K_{I+1,j} f_j^N \nonumber \\
    &-\left(\Phi^{\Lambda}\right)'\left( f_I^N+\eta\left(f_{I+1}^N+f_{I-1}^N\right)\right) f_{I+1}^N \sum_{j=0}^{\infty} K_{I+1,j} f_j\Bigg|\, dx\, dt \nonumber \\
    & \lesssim \sup_{j\geq N^*}\left\{ \frac{K_{I+1,j}}{j+1}\right\} \to 0, \ \text{as}\ N^*\to \infty \ (\text(see)\ \eqref{H_kernel}).
\end{align*}
The above calculations along with \eqref{truncated to original, limit pass, intermediate 2}, yields that
\begin{align*}
\int_0^T \int_{\Omega} &\left(\Phi^{\Lambda}\right)'\left( f_I^N+\eta\left(f_{I+1}^N+f_{I-1}^N\right)\right) \mathcal{L}_1^N\\
&= \int_0^T \int_{\Omega} \left(\Phi^{\Lambda}\right)'\left( f_I+\eta\left(f_{I+1}+f_{I-1}\right)\right) f_{I+1} \sum_{j=0}^{\infty} K_{I+1,j} f_j\, dx\, dt.
\end{align*}
Similar calculation holds for the terms involving $\mathcal{L}_2^N, \mathcal{L}_3^N, \mathcal{L}_4^N$ also. Adding all the terms we conclude that
\begin{align*}
\int_0^T \int_{\Omega} &\left(\Phi^{\Lambda}\right)'\left( f_I^N+\eta\left(f_{I+1}^N+f_{I-1}^N\right)\right) Q_I^N(f^N) \, dx\, dt\\
&= \int_0^T \int_{\Omega} \left(\Phi^{\Lambda}\right)'\left( f_I+\eta\left(f_{I+1}+f_{I-1}\right)\right) Q_I(f) \, dx\, dt.
\end{align*}
Thanks to \eqref{uniform bound source in N}, we have that 
\[
\int_0^T \int_{\Omega} \left| Q_i^N(f^N)\right| \, dx\, dt \lesssim \sup_{i,j}\left\{\frac{K_{i,j}}{(i+1)(j+1)}\right\} \left( \frac{d^*}{d_*} \right)
\left\|
\sum_{i=0}^{\infty} (i+1)f_{i,0}
\right\|_{L^2(\Omega)}^{2},
\]
for all $i\geq 0$. Thus, we have that
\[
\left|\int_0^T \int_{\Omega} \left(\Phi^{\Lambda}\right)'\left( f_I^N+\eta\left(f_{I+1}^N+f_{I-1}^N\right)\right) \eta Q_{I\pm 1}^N(f^N) \xi \, dx\, dt\right|\lesssim \eta \| \xi\|_{L^{\infty}((0,T)\times\Omega)}. 
\]
Hence, we handled the third term of Proposition \ref{TS: composition with a special truncation to identity}. Handling the fourth term is  straightforward. We obtain the following
\begin{align*}
  \int_0^T \int_{\Omega} &d_I (\Phi^{\Lambda})' \left(f_I^N+\eta\left(f_{I+1}^N+f_{I-1}^N\right)\right) \left(\nabla f_I^N+\eta\left(\nabla f_{I+1}^N+\nabla f_{I-1}^N\right)\right)\cdot \nabla \xi \, dx\, dt\\
  &=\int_0^T \int_{\Omega} d_I (\Phi^{\Lambda})' \left(f_I+\eta\left(f_{I+1}+f_{I-1}\right)\right) \left(\nabla f_I+\eta\left(\nabla f_{I+1}+\nabla f_{I-1}\right)\right) \cdot\nabla \xi \, dx\, dt.
\end{align*}
Let us move onto the fifth term of Proposition \ref{TS: composition with a special truncation to identity}, the term reads as
\[
-\int_0^T\int_{\Omega} d_{I}(\Phi^{\Lambda})'' \left(f_I^N+\eta\left(f_{I+1}^N+f_{I-1}^N\right)\right) \left| \nabla f_I^N+\eta\left(\nabla f_{I+1}^N+\nabla f_{I-1}^N\right)\right|^2 \xi \, dx\, dt.
\]
Let us define the following radon measure on $([0,T)\times\Omega)$:
\[
\nu_{I,\eta}^{\Lambda, N}:= d_{I}(\Phi^{\Lambda})'' \left(f_I^N+\eta\left(f_{I+1}^N+f_{I-1}^N\right)\right) \left| \nabla f_I^N+\eta\left(\nabla f_{I+1}^N+\nabla f_{I-1}^N\right)\right|^2 \, dx\, dt.
\]
Let us define $F^N:= f_I^N+\eta\left( f_{I+1}^N+f_{I-1}^N\right)$. The following estimate holds
\begin{align*}
    \int_0^T\int_{\Omega} &\frac{\left|\nabla {F^N}\right|^2}{F^N}\, dx \, dt \leq 3\int_0^T\int_{\Omega} \frac{\left|\nabla {f_I^N}\right|^2}{f_I^N+\eta\left( f_{I+1}^N+f_{I-1}^N\right)}\, dx \, dt\\
    &+ 3\int_0^T\int_{\Omega} \frac{\left|\eta\nabla {f_{I+1}^N}\right|^2}{f_I^N+\eta\left( f_{I+1}^N+f_{I-1}^N\right)}\, dx \, dt+3\int_0^T\int_{\Omega} \frac{\left|\eta\nabla {f_{I-1}^N}\right|^2}{f_I^N+\eta\left( f_{I+1}^N+f_{I-1}^N\right)}\, dx \, dt\\
    &\leq 3\int_0^T\int_{\Omega} \frac{\left|\nabla {f_{I}^N}\right|^2}{f_I^N}\, dx \, dt+ 3\int_0^T\int_{\Omega} \frac{\eta\left|\nabla {f_{I+1}^N}\right|^2}{f_{I+1}^N}\, dx \, dt\\
    &+3\int_0^T\int_{\Omega} \frac{\eta\left|\nabla {f_{I-1}^N}\right|^2}{f_{I-1}^N}\, dx \, dt \leq 3(1+2\eta)M_{\mathcal{F}}.
\end{align*}
Thus $\displaystyle{\left| \nu_{I,\eta}^{\Lambda,N}\right| ([0,T)\times\Omega)\leq 3\sup_{i\geq 0}\{d_i\}K_2(1+2\eta)M_{\mathcal{F}}}$, where $K_2$ is as defined in \ref{truncation to identity}. Hence, we conclude that $\nu^{\Lambda,N}_{I,\eta}\overset{*}{\rightharpoonup} \nu^{\Lambda}_{I,\eta}$. It yields the following limit
\begin{align*}
   \lim_{N\to\infty}-\int_0^T\int_{\Omega}& d_{I}(\Phi^{\Lambda})'' \left(f_I^N+\eta\left(f_{I+1}^N+f_{I-1}^N\right)\right) \left| \nabla f_I^N+\eta\left(\nabla f_{I+1}^N+\nabla f_{I-1}^N\right)\right|^2 \xi \, dx\, dt \\
   =& -\int_0^T \int_{\Omega} \nu_{I,\eta}^{\Lambda}\xi\, dx\, dt.
\end{align*}
Next we consider the following positive measures
\[
\mu^{K,N}:= \chi_{\left\{\left|F^N\right|\in[K-1,K)\right\}}\left| \nabla \sqrt{F^N}\right|^2 \, dx\, dt,
\]
on $[0,T)\times\Omega$. We calculate the total variation of $\nu^{\Lambda,N}_{I, \eta}$. The following computation holds
\begin{align}
 |\nu^{\Lambda,N}_{I,\eta}|&([0,T)\times\Omega) \label{TS: second intermediate measure}\\ 
 &= \int_0^T\int_{\Omega}  4d_I \left|F^N \left(\Phi^{\Lambda}\right)''\left(F^N\right) \left|\nabla \sqrt{\left(F^N\right)} \right|^2 \right| \, dx\, dt   \nonumber\\
 &\leq 4\sup_{j\geq 0}\{d_j\} \times \nonumber  \sum_{K=1}^{\infty}\int_0^T\int_{\Omega}   \chi_{\left\{\left|f_{j,\epsilon}^N\right|\in[K-1,K)\right\}}\left|F^N\left(\Phi^{\Lambda}\right)''\left(F^N\right)\right| \left|\nabla \sqrt{\left(f^N\right)} \right|^2  \, dx\, dt   \nonumber\\
 & \leq 4\sup_{j\geq 0}\{d_j\} \sum_{K=1}^{\infty} \mu^{K,N} ([0,T)\times\Omega)\cdot \sup_{|v|\in[K-1,K); v\geq 0} |v|\left| \left(\Phi^{\Lambda}\right)''(v)\right|. \nonumber
\end{align}
Note that, thanks to \ref{truncation to identity}, only finitely terms survive in the sum. Let till $K=K_{max}<+\infty$ the sum survives. We have that 
\[
\sum_{K=1}^{K_{max}} \mu^{K,N}([0,T)\times\Omega)\leq \int_0^T\int_{\Omega} \left| \nabla \sqrt{F^N}\right|^2\, dx\, dt \leq 3(1+2\eta)M_{\mathcal{F}}.
\]
This further yields
\begin{align} \label{TS: third intermediate measure}
   3(1+2\eta) M_{\mathcal{F}}  \geq \limsup_{\epsilon\to 0}\sum_{K=1}^{K_{max}} \mu^{K,N}([0,T)\times\Omega).
\end{align}
The boundednes and nonnegativity of $\mu^{K,N}$ implies that there exists a subsequence (we still index it by $N$)  such that $\displaystyle{\lim_{N\to \infty}\mu^{K,N}([0,T)\times\Omega)}$ exists. Hence the above relation \eqref{TS: third intermediate measure} can be written as
\begin{align} \label{TS: third intermediate measure new 1}
     3(1+2\eta)M_{\mathcal{F}}  \geq \sum_{K=1}^{K_{max}} \lim_{N\to \infty}\mu^{K,N}([0,T)\times\Omega).
\end{align}
Using the fact that measure of open sets is lower semicontinuous with respect to weak-* convergence and \eqref{TS: second intermediate measure}, we obtain the following
\begin{align*}
    &\left| \nu_{I,\eta}^{\Lambda}\right|  ([0,T)\times\Omega)\leq \liminf_{N\to\infty} \left| \nu_{I,\eta}^{\Lambda, N} \right|([0,T)\times\Omega)\\
    &  \leq 4\sup_{j\geq 0}\{d_j\} \lim_{N\to\infty}\sum_{K=1}^{\infty} \mu^{K,N} ([0,T)\times\Omega)\cdot \sup_{|v|\in[K-1,K); v\geq 0} |v|\left| \left(\Phi^{\Lambda}\right)''(v)\right|\\
    &\leq 4\sup_{j\geq 0}\{d_j\} \sum_{K=1}^{K_{max}} \lim_{N\to\infty}\mu^{K,N} ([0,T)\times\Omega)\cdot \sup_{|v|\in[K-1,K); v\geq 0} |v|\left| \left( \Phi^{\Lambda}\right)''(v)\right|\\
    & \leq  4\sup_{j\geq 0}\{d_j\} \sum_{K=1}^{\infty}\liminf_{N\to\infty} \mu^{K,N} ([0,T)\times\Omega)\cdot \sup_{|v|\in[K-1,K); v\geq 0} |v|\left| \left( \Phi^{\Lambda}\right)''(v)\right|
\end{align*}
where we use \eqref{TS: third intermediate measure new 1} in the third line. We have that
\[
\left | \liminf_{N\to \infty} \mu^{K,N} ([0,T)\times\Omega)\cdot \sup_{|v|\in[K-1,K)} |v|\left| \left( \Phi^{\Lambda}\right)''(v)\right| \right| \leq K_1 \liminf_{N\to\infty} \mu^{K,N} ([0,T)\times\Omega),
\]
where 
\[
\sum_{K=1}^{\infty} \liminf_{N\to\infty} \mu^{K,N} ([0,T)\times\Omega) \leq \liminf_{N\to\infty}\sum_{K=1}^{\infty} \mu^{K,N} ([0,T)\times\Omega) \leq 3(1+2\eta)M_{\mathcal{F}}
\]
thanks to Fatou's lemma. Hence using dominated convergence theorem on counting measure, we obtain that
\begin{align*}
&\lim_{\Lambda\to\infty}\left| \nu_{I,\eta}^{\Lambda}\right| ([0,T)\times\Omega) \\
& \leq  4\sup_{j\geq 0}\{d_j\} \sum_{K=1}^{\infty} \liminf_{N\to\infty} \mu^{K,N} ([0,T)\times\Omega)\cdot \lim_{\Lambda\to\infty}\sup_{|v|\in[K-1,K);v\geq 0} v\left| \left(\Phi^{\Lambda}\right)''(v)\right|=0.
\end{align*}
Rest of the proof relies on showing that the rest of the terms in Proposition \ref{TS: composition with a special truncation to identity} are of the order $\eta^{\frac 12}$ when $\eta\in(0,1)$. We present our calculation for a particular term. For rest of the terms, the calculations are similar. Consider  the term
\[
\int_0^T\int_{\Omega}\eta (d_{I-1}-d_I) \left(\Phi^{\Lambda}\right)''\left(f_I^N+\eta\left( f_{I+1}^N+f_{I-1}^N\right)\right) \nabla f_{I-1}^N\cdot\nabla f_{I}^N\xi \, dx\, dt.
\]
Note that the above term in non zero only when $\displaystyle{f_I^N+\eta\left(f_{I+1}^N+f_{I-1}^N\right)\leq \Lambda^*}$, where $\Lambda^*$ is as defined in \ref{truncation to identity}. Thanks to Lemma \ref{lem:trunc-energy} and the fact that, for $0\leq i\leq N$, $f_i^N$ are the solution to the system \eqref{truncated system}, we have that 
\begin{align*}
    d_I &\int_{\left\{| f^N_I|\leq \Lambda^*\right\}} | \nabla f_I^N|^2 \, dx\, dt \leq \Lambda^* \left( \int_0^T\int_{\Omega} \left| Q_I^N(f^N)\right| \, dx\, dt+ \int_{\Omega} | f_{I,0}| \, dx\right)\\
    \lesssim & \Lambda^* \left( \sup_{i,j}\left\{\frac{K_{i,j}}{(i+1)(j+1)}\right\} \left( \frac{d^*}{d_*} \right)
\left\|
\sum_{i=0}^{\infty} (i+1)f_{i,0}
\right\|_{L^2(\Omega)}^{2} +\int_{\Omega} | f_{I,0}|\, dx\right),\\
  d_{I-1} &\int_{\left\{| f^N_{I-1}|\leq \frac{\Lambda^*}{\eta}\right\}} | \nabla f_{I-1}^N|^2 \, dx\, dt \leq \frac{\Lambda^*}{\eta} \left( \int_0^T\int_{\Omega} \left| Q_{I-1}^N(f^N)\right| \, dx\, dt+ \int_{\Omega} | f_{I-1,0}| \, dx\right)\\
    \lesssim & \frac{\Lambda^*}{\eta} \left( \sup_{i,j}\left\{\frac{K_{i,j}}{(i+1)(j+1)}\right\} \left( \frac{d^*}{d_*} \right)
\left\|
\sum_{i=0}^{\infty} (i+1)f_{i,0}
\right\|_{L^2(\Omega)}^{2} +\int_{\Omega} | f_{I-1,0}|\, dx\right),
\end{align*}
where we use the estimate \eqref{uniform bound source in N} to estimate the source term. We have the following estimate:
\begin{align*}
    \Bigg|\int_0^T&\int_{\Omega}\eta (d_{I-1}-d_I) \left(\Phi^{\Lambda}\right)''\left(f_I^N+\eta\left( f_{I+1}^N+f_{I-1}^N\right)\right) \nabla f_{I-1}^N\cdot\nabla f_{I}^N\xi \, dx\, dt\Bigg| \\
    \leq & 2\eta \sup_{j\geq 0}\{d_j\}\sup_{v\geq 0}\left| \left(\Phi^{\Lambda}\right)''(v)\right|\left(\int_{\left\{|f_{I}^N|\leq \Lambda^*\right\}} \left| \nabla f_I^N\right|^2 \, dx\, dt\right)^{\frac 12}\left(\int_{\left\{|f_{I-1}^N|\leq \frac{\Lambda^*}{\eta}\right\}} \left|\nabla f_{I-1}^N\right|^2 \, dx\, dt\right)^{\frac 12}\\
    &  \lesssim 2\eta \sup_{j\geq 0}\{d_j\}\sup_{v\geq 0}\left| \left(\Phi^{\Lambda}\right)''(v)\right| \cdot \left(\Lambda^*\right)^{\frac 12}\cdot \left(\frac{\Lambda^*}{\eta}\right)^{\frac{1}{2}} \lesssim \eta^{\frac{1}{2}}.
\end{align*}
Thus, passing to the limit \(N \to \infty\) in each term in Proposition \ref{TS: composition with a special truncation to identity}, we conclude the proof of the lemma.
\end{proof}
 In the next lemma we let the limit $\eta \to 0$.
 \begin{lem}\label{recovering the single equation}
      For $i\in \mathbb{N}\cup\{0\}$, let $f_i$ be the function  described in Lemma \ref{compactness to the solution of TS}. Let $\Phi^{\Lambda}$ be as defined in \ref{truncation to identity}. Let us fix an index $i=I$. Then for all $\xi\in C_c^{\infty}([0,T)\times\Omega)$, the following holds
    \begin{align*}
       & -\int_{\Omega} \Phi^{\Lambda} \left(f_{I,0}\right) \xi(0)\, dx-\int_0^T \int_{\Omega} \Phi^{\Lambda} \left( f_I \right) \partial_t \xi\, dx\, dt -\int_0^T \int_{\Omega} \left(\Phi^{\Lambda}\right)'\left( f_I\right) Q_I \xi \, dx\, dt\\
       &+ \int_0^T \int_{\Omega} \left(\Phi^{\Lambda}\right)'\left( f_I\right) \nabla f_I \cdot\nabla \xi \, dx\, dt + \int_0^T\int_{\Omega} \nu_{I}^{\Lambda} \xi \, dx\, dt=0,
    \end{align*}
    for some radon measure $\nu_{I}^{\Lambda}\in RM([0,T)\times\Omega)$. Here $Q_I(f)$ is as defined in \eqref{main equation} and 
    \[
    \lim_{\Lambda\to \infty} \left|\nu_{I}^{\Lambda} \right|([0,T)\times\Omega)\to 0, \quad \forall I\in\mathbb{N}.
    \]
 \end{lem}
\begin{proof}
    We let $\eta\to 0+$ in Lemma \ref{N to infinity inside truncation}. Rest of the proof is similar to Lemma \ref{N to infinity inside truncation}.
\end{proof}
Now we are well equipped to derive the main result of this article.



\begin{proof}[Proof of Theorem \ref{main result}:]

We let $\Lambda\to \infty$ in Lemma \ref{recovering the single equation}. Note that the fifth term in Lemma \ref{recovering the single equation} will vanish as $\Lambda\to\infty$. The calculations for the first and second term are similar. We only present the calculation for the second term. We have that
\begin{align*}
    \lim_{\Lambda\to \infty}\int_0^T \int_{\Omega} \Phi^{\Lambda} &\left( f_I \right) \partial_t \xi\, dx\, dt = \lim_{\Lambda\to \infty}\int_{\left\{f_I\leq \Lambda\right\}} \Phi^{\Lambda} \left( f_I \right) \partial_t \xi\, dx\, dt\\
    +& \lim_{\Lambda\to \infty}\int_{\left\{f_I> \Lambda\right\}} \left(\Phi^{\Lambda}\left( f_I \right)-\Phi(0)\right)  \partial_t \xi\, dx\, dt\\
    & \lim_{\Lambda\to \infty}\int_{\left\{f_I> \Lambda\right\}} \Phi(0)  \partial_t \xi\, dx\, dt:= J_1+J_2+J_3.
\end{align*}
We use the fact that $\Phi^{\Lambda}(v)=v$, when $v\leq \Lambda$ (see \ref{truncation to identity}). It yields that
\[
J_1= \lim_{\Lambda\to \infty}\int_{\left\{f_I\leq \Lambda\right\}}  f_I  \partial_t \xi\, dx\, dt= \lim_{\Lambda\to \infty}\int_{\Omega} \chi_{\left\{f_I\leq \Lambda\right\}} f_I  \partial_t \xi\, dx\, dt.
\]
Thanks to dominated convergence theorem we have that
\[
J_1= \int_{\Omega} f_{I}\partial_t \xi \, dx\, dt.
\]
Note that since $\displaystyle{\int_{\Omega} f_I \, dx<+\infty}$, we have that $f_I(x)$ is finite almost everywhere. Hence using dominated convergence theorem again we have 
\[
J_3=0.
\]
Finally, we rewrite the term $J_2$ as
\begin{align*}
    J_2= \lim_{\Lambda\to\infty} \int_{\Omega} \chi_{\left\{f_I>\Lambda\right\}} \left(\Phi^{\Lambda} (f^I)-\Phi(0)\right) \partial_t\xi\, dx\, dt.
\end{align*}
The integrand converges pointwise to $0$ and dominated by the function $\displaystyle{K_2}f^I\| \partial_t\xi\|_{L^{\infty}((0,T)\times\Omega)}$ (see \ref{truncation to identity}). Hence, using dominated convergence theorem we obtain that
\[
J_3=0,
\]
which further yields
\[
\lim_{\Lambda\to \infty}\int_0^T \int_{\Omega} \Phi^{\Lambda} \left( f_I \right) \partial_t \xi\, dx\, dt = \int_0^T \int_{\Omega} f_I  \partial_t \xi\, dx\, dt. 
\]
Similarly, for the first term we have that
\[
\lim_{\Lambda\to\infty}\int_{\Omega} \Phi^{\Lambda} \left(f_{I,0}\right) \xi(0)\, dx= \int_{\Omega} f_{I,0} \xi(0)\, dx.
\]
Passing limits to third and fourth term is an application of dominated convergence theorem. We have that
\begin{align*}
&\lim_{\Lambda\to\infty}\int_0^T \int_{\Omega} \left(\Phi^{\Lambda}\right)'\left( f_I\right) Q_I \xi \, dx\, dt =\int_0^T \int_{\Omega}  Q_I \xi \, dx\, dt,\\
& \lim_{\Lambda\to\infty}\int_0^T \int_{\Omega} \left(\Phi^{\Lambda}\right)'\left( f_I\right) \nabla f_I \cdot \nabla\xi \, dx\, dt =\int_0^T \int_{\Omega}  \nabla f_I \cdot \nabla\xi \, dx\, dt,
\end{align*}
where we use the fact that $\displaystyle{\lim_{\Lambda\to\infty} \frac{d}{dv}\Phi^{\Lambda}(v)=1}$ (see \ref{truncation to identity}). Combining all, passing $\Lambda\to \infty$ in Lemma \ref{recovering the single equation}, we obtain that
\[
-\int_{\Omega} f_{I,0}\xi(0)\, dx-\int_0^T\int_{\Omega} f_I\partial_t \xi\, dx\, dt-\int_0^T\int_{\Omega} Q_I(f) \xi\, dx\, dt+\int_0^T\int_{\Omega} \nabla f_I\cdot\nabla \xi\, dx\, dt=0
\]
\end{proof}

\medskip
\textbf{Acknowledgment:}  The first author received a Postdoctoral Research Fellowship
from the Department of Atomic Energy (DAE), Government of India. The second author gratefully acknowledges support from the Anusandhan National Research Foundation (ANRF), India, through the National Post-Doctoral Fellowship (NPDF) [File No.
PDF/2025/007779].


\vspace{.3cm}
\textbf{Data Availability:} The authors shall permit all the data underlying the findings of this manuscript to be shared
by any researchers or groups who are interested in the article.

\vspace{.3cm}
\textbf{Declaration:}

\vspace{.3cm}
\textbf{Conflict of Interest:} The authors declare that there is no conflict of interest regarding the publication of this paper.


\noindent

\medskip
\bibliography{Refs.bib}

\begin{thebibliography}{10}

\bibitem{amann1995linear}
H.~Amann et~al.
\newblock {\em Linear and quasilinear parabolic problems}, volume~1.
\newblock Springer, 1995.

\bibitem{amann2005local}
H.~Amann and C.~Walker.
\newblock Local and global strong solutions to continuous
  coagulation--fragmentation equations with diffusion.
\newblock {\em Journal of Differential Equations}, 218(1):159--186, 2005.

\bibitem{bll2019}
J.~Banasiak, W.~Lamb, and {\relax Ph}.~Lauren{\c{c}}ot.
\newblock {\em Analytic methods for coagulation-fragmentation models}.
\newblock CRC Press, 2019.

\bibitem{baras1984}
P.~Baras and M.~Pierre.
\newblock Problems paraboliques semi-lineaires avec donnees measures.
\newblock {\em Applicable Analysis}, 18(1-2):111--149, 1984.

\bibitem{Naim2003}
E.~Ben-Naim and P.~L. Krapivsky.
\newblock Exchange-driven growth.
\newblock {\em Physical Review E}, 68(3):031104, 2003.

\bibitem{bothe2010}
D.~Bothe and M.~Pierre.
\newblock Quasi-steady-state approximation for a reaction--diffusion system
  with fast intermediate.
\newblock {\em Journal of Mathematical Analysis and Applications},
  368(1):120--132, 2010.

\bibitem{canizo2010absence}
J.~Canizo, L.~Desvillettes, and K.~Fellner.
\newblock Absence of gelation for models of coagulation-fragmentation with
  degenerate diffusion.
\newblock {\em Il Nuovo cimento della Societ{\`a} italiana di fisica. C},
  33(1):79, 2010.

\bibitem{canizo2010regularity}
J.~A. Ca{\~n}izo, L.~Desvillettes, and K.~Fellner.
\newblock Regularity and mass conservation for discrete
  coagulation-fragmentation equations with diffusion.
\newblock {\em Annales de l'IHP Analyse non lin{\'e}aire}, 27(2):639--654,
  2010.

\bibitem{canizo2014improved}
J.~A. Canizo, L.~Desvillettes, and K.~Fellner.
\newblock Improved duality estimates and applications to reaction-diffusion
  equations.
\newblock {\em Communications in Partial Differential Equations},
  39(6):1185--1204, 2014.

\bibitem{das2025existence}
S.~Das.
\newblock Existence of solution of a triangular degenerate reaction--diffusion
  system: S. das.
\newblock {\em Journal of Evolution Equations}, 25(2):51, 2025.

\bibitem{das2026existence}
S.~Das and R.~G. Jaiswal.
\newblock Existence for the discrete nonlinear fragmentation equation with
  degenerate diffusion.
\newblock {\em arXiv preprint arXiv:2602.14070}, 2026.

\bibitem{desvillettes2015duality}
L.~Desvillettes and K.~Fellner.
\newblock Duality and entropy methods for reversible reaction-diffusion
  equations with degenerate diffusion.
\newblock {\em Mathematical Methods in the Applied Sciences},
  38(16):3432--3443, 2015.

\bibitem{desvillettes2007global}
L.~Desvillettes, K.~Fellner, M.~Pierre, and J.~Vovelle.
\newblock Global existence for quadratic systems of reaction-diffusion.
\newblock {\em Advanced Nonlinear Studies}, 7(3):491--511, 2007.

\bibitem{ES2021}
C.~Eichenberg and A.~Schlichting.
\newblock Self-similar behavior of the exchange-driven growth model with
  product kernel.
\newblock {\em Communications in Partial Differential Equations},
  46(3):498--546, 2021.

\bibitem{E2018}
E.~Esenturk.
\newblock Mathematical theory of exchange-driven growth.
\newblock {\em Nonlinearity}, 31(8):3460--3483, 2018.

\bibitem{EV2021}
E.~Esenturk and J.~J.~L. Vel{\'a}zquez.
\newblock Large time behavior of exchange-driven growth.
\newblock {\em Discrete and Continuous Dynamical Systems}, 41(2):747--775,
  2021.

\bibitem{fellner2021uniform}
K.~Fellner, J.~Morgan, and B.~Q. Tang.
\newblock Uniform-in-time bounds for quadratic reaction-diffusion systems with
  mass dissipation in higher dimensions.
\newblock {\em Discrete \& Continuous Dynamical Systems-Series S}, 14(2), 2021.

\bibitem{JF2015}
J.~Fischer.
\newblock Global existence of renormalized solutions to entropy-dissipating
  reaction--diffusion systems.
\newblock {\em Archive for Rational Mechanics and Analysis}, 218(1):553--587,
  2015.

\bibitem{fitzgibbon2021reaction}
W.~E. Fitzgibbon, J.~J. Morgan, B.~Q. Tang, and H.-M. Yin.
\newblock Reaction-diffusion-advection systems with discontinuous diffusion and
  mass control.
\newblock {\em SIAM Journal on Mathematical Analysis}, 53(6):6771--6803, 2021.

\bibitem{IKR1998}
S.~Ispolatov, P.~L. Krapivsky, and S.~Redner.
\newblock Wealth distributions in models of capital exchange.
\newblock {\em The European Physical Journal B}, 2:267--276, 1998.

\bibitem{JZ2003}
J.~Ke and Z.~Lin.
\newblock Kinetics of migration-driven aggregation processes with birth and
  death.
\newblock {\em Physical Review E}, 67(3):031103, 2003.

\bibitem{laurencot2002continuous}
P.~Lauren{\c{c}}ot and S.~Mischler.
\newblock The continuous coagulation-fragmentation equations with diffusion.
\newblock {\em Archive for rational mechanics and analysis}, 162(1):45--99,
  2002.

\bibitem{laurenccot2002}
P.~Lauren{\c{c}}ot and S.~Mischler.
\newblock Global existence for the discrete diffusive coagulation-fragmentation
  equations in $ l^{1}$.
\newblock {\em Rev. Mat. Iberoamericana}, 18(1):731--745, 2002.

\bibitem{LR2002}
F.~Leyvraz and S.~Redner.
\newblock Scaling theory for migration-driven aggregate growth.
\newblock {\em Physical Review Letters}, 88(6):068301, 2002.

\bibitem{list1976}
R.~List and J.~R. Gillespie.
\newblock Evolution of raindrop spectra with collision-induced breakup.
\newblock {\em Journal of the Atmospheric Sciences}, 33(10):2007--2013, 1976.

\bibitem{pierre2003weak}
M.~Pierre.
\newblock Weak solutions and supersolutions in for reaction-diffusion systems.
\newblock {\em Journal of Evolution Equations}, 3(1):153--168, 2003.

\bibitem{Pierre2010}
M.~Pierre.
\newblock Global existence in reaction-diffusion systems with control of mass:
  a survey.
\newblock {\em Milan Journal of Mathematics}, 78(2):417--455, 2010.

\bibitem{quittner2007}
P.~Quittner and P.~Souplet.
\newblock {\em Superlinear parabolic problems: blow-up, global existence and
  steady states}.
\newblock Springer, 2007.

\bibitem{rothe2006}
F.~Rothe.
\newblock {\em Global solutions of reaction-diffusion systems}.
\newblock Springer, 2006.

\bibitem{safronov1972}
V.~S. Safronov.
\newblock {\em Evolution of the Protoplanetary Cloud and Formation of the Earth
  and the Planets}.
\newblock Israel Program for Scientific Translations, Jerusalem, 1972.
\newblock English translation of the 1969 Russian edition.

\bibitem{S2020}
A.~Schlichting.
\newblock The exchange-driven growth model: Basic properties and longtime
  behavior.
\newblock {\em Journal of Nonlinear Science}, 30(3):793--830, 2020.

\bibitem{si2025}
S.~Si and A.~K. Giri.
\newblock Existence and non-existence for exchange-driven growth model.
\newblock {\em Nonlinearity}, 38:125014, 2025.

\bibitem{walker2004discrete}
C.~Walker.
\newblock The discrete diffusive coagulation--fragmentation equations with
  scattering.
\newblock {\em Nonlinear Analysis: Theory, Methods \& Applications},
  58(1-2):121--142, 2004.

\bibitem{walker2005new}
C.~Walker.
\newblock {On a new model for continuous coalescence and breakage processes
  with diffusion}.
\newblock {\em Advances in Differential Equations}, 10(2):121 -- 152, 2005.

\bibitem{wrzosek2002}
D.~Wrzosek.
\newblock Mass-conserving solutions to the discrete coagulation--fragmentation
  model with diffusion.
\newblock {\em Nonlinear Analysis: Theory, Methods \& Applications},
  49(3):297--314, 2002.

\end{thebibliography}
\bibliographystyle{abbrv}
\end{document}